%% file: FFM.tex
\documentclass[11pt]{amsart}

\usepackage{times,amsfonts,amsmath,amstext,amsbsy,amssymb,
  amsopn,amsthm,upref,eucal, amscd, color, graphicx}
\usepackage[T1]{fontenc}
\usepackage[stretch=10]{microtype}

\usepackage{hyperref}
\usepackage{color}
\definecolor{darkred}{rgb}{0.4,0,0}
\definecolor{darkgreen}{rgb}{0,0.5,0}
\definecolor{darkblue}{rgb}{0,0,0.4}

\hypersetup{
    pdftitle={},	
    pdfauthor={Filip, Forni, Matheus},	
    pdfsubject={math.DS},		
    pdfkeywords={origami, flat surface, Lyapunov exponent},	
    pdfnewwindow=true,		
    colorlinks=true,		
    linkcolor=darkblue,		
    citecolor=darkred,		
    filecolor=darkblue,		
    urlcolor=darkblue,		
    pdfborder={0 0 0},
    breaklinks=true
}

\makeatletter
\renewcommand{\paragraph}{%
\@startsection {paragraph}{4}
{\z@} \z@ {-\fontdimen 2\font }\bfseries
}
\makeatother

\newcommand{\ip}[1]{\left\langle#1\right\rangle}
\newcommand{\Hom}{\operatorname{Hom}}
\newcommand{\Res}{\operatorname{Res}}
\newcommand{\Ind}{\operatorname{Ind}}

\usepackage{titletoc}

\let\oldtocsection=\tocsection
\let\oldtocsubsection=\tocsubsection
\let\oldtocsubsubsection=\tocsubsubsection

\renewcommand{\tocsection}[2]{\hspace{0em}\oldtocsection{#1}{#2}}
\renewcommand{\tocsubsection}[2]{\hspace{1.75em}\oldtocsubsection{#1}{#2}}
\renewcommand{\tocsubsubsection}[2]{\hspace{2em}\oldtocsubsubsection{#1}{#2}}

\def\subsek~{\S{}}

\def\equationautorefname~#1\null{%
  Equation~(#1)\null
}

\newtheorem{theorem}{Theorem}[section]

\numberwithin{equation}{section}

\theoremstyle{definition}
\newtheorem{definition}[theorem]{Definition}

\newtheorem{remark}[theorem]{Remark}

\newcommand{\tr}{{\text{\rm tr}\,}}

\begin{document}
\title[Exotic monodromies and quaternionic covers]{Quaternionic covers and monodromy of the Kontsevich-Zorich cocycle in orthogonal groups}

\author{Simion Filip}
\address{Simion Filip: Department of Mathematics, University of Chicago, Chicago, IL 60615, USA}
\email{sfilip@math.uchicago.edu}

\author{Giovanni Forni}
\address{Giovanni Forni: Department of Mathematics, University of Maryland, College Park, MD 20742-4015, USA}
\email{gforni@math.umd.edu}

\author{Carlos Matheus}
\address{Carlos Matheus: Universit\'e Paris 13, Sorbonne Paris Cit\'e, LAGA, CNRS (UMR 7539), F-93439, Villetaneuse, France}
\email{matheus@impa.br.}

\date{\today}

\begin{abstract}
We give an example of a Teichm\"uller curve which contains, in a factor of its monodromy, a group which was not observed before.
Namely, it has Zariski closure equal to the group $SO^*(6)$ in its standard representation; up to finite index, this is the same as $SU(3,1)$ in its second exterior power representation.

The example is constructed using origamis (i.e. square-tiled surfaces). It can be generalized to give monodromy inside the group $SO^*(2n)$ for all $n$, but in the general case the monodromy might split further inside the group.

Also, we take the opportunity to compute the multiplicities of representations in the (0,1) part of the cohomology of regular origamis, answering a question of Matheus-Yoccoz-Zmiaikou.

%
\end{abstract}

\maketitle

\tableofcontents


\section{Introduction}\label{s.introduction}

A \emph{translation surface} is the data $(M,\omega)$ of a compact Riemann surface $M$ of genus $g\geq 1$ and a non-zero Abelian differential (holomorphic $1$-form) $\omega$ on $M$.

The moduli spaces of translation surfaces possess a natural $SL(2,\mathbb{R})$-action whose dynamical features play a key role in some applications to Dynamical Systems and Algebraic Geometry such as the study of interval exchange transformations and translation flows, and the  classification of commensurability classes of ball quotients introduced by Deligne and Mostow. See, for example, the works of Masur \cite{Masur1}, Veech \cite{Veech1}, Zorich \cite{Zorich1}, Forni \cite{Forni}, Delecroix-Hubert-Leli\`evre \cite{DHL}, Athreya-Eskin-Zorich \cite{AEZ}, and Kappes-M\"oller \cite{KM} for some illustrations.

A major actor in the investigation of the dynamics of the $SL(2,\mathbb{R})$-action on moduli spaces of translation surfaces is the so-called Kontsevich-Zorich cocycle (KZ cocycle for short): for instance, the properties of the KZ cocycle are a crucial ingredient in the celebrated recent work of Eskin-Mirzakhani \cite{EM} towards the classification of $SL(2,\mathbb{R})$-invariant measures on moduli spaces of translation surfaces.


A detailed study of the KZ cocycle was initiated by Forni \cite{Forni, Fo11}.
These works are concerned with formulas for exponents, spectral gap and non-triviality of exponents, as well as applications to dynamics on translation surfaces.
Later, in \cite{Filip1} it was proved that the KZ cocycle is semisimple and its decomposition respects the Hodge structure.
Using this property, in \cite{Filip2}, an analysis of possible groups appearing in the Zariski closure of the monodromy (or the algebraic hull) was done.
Up to finite index and compact factors, the list of groups and representations is:
\begin{itemize}
\item[(i)] $Sp(2d,\mathbb{R})$ in the standard representation;
\item[(ii)] $SU_{\mathbb{C}}(p,q)$ in the standard representation;
\item[(iii)] $SU_{\mathbb{C}}(p,1)$ in an exterior power representation;
\item[(iv)] $SO^*(2n)$ in the standard representation (see \cite{Helgason} and \autoref{subsec:SOstar} for a description);
\item[(v)] $SO_{\mathbb{R}}(n,2)$ in a spin representation.
\end{itemize}
In other words, this list of five representations \emph{suffices} to account for all possibilities for the monodromy group of the KZ cocycle.

Nevertheless, our understanding of the KZ cocycle is not completely satisfactory yet: for instance, while the monodromy groups in items (i) and (ii) above appear in several families of examples in the literature (see, e.g., Eskin, Kontsevich, and Zorich \cite{EKZ2} and McMullen \cite{McMullen}), Question 5.5 of \cite{Filip2} asks whether the monodromy groups in items (iii), (iv) and (v) \emph{actually} occur in the context of the KZ cocycle
\footnote{\label{foot.2} It is known that each item of this list can be realized \emph{abstractly} as monodromy group of variations of Hodge structures over certain families of Riemann surfaces and Abelian varieties (see \cite{Filip2} for more explanations).
However, it is not clear how to convert these abstract realizations into non-compact factors of the KZ cocycle over the closure of some $SL(2,\mathbb{R})$-orbit in the moduli space of translation surfaces.}.

In this note, we give the following partial answer to this question.

\begin{theorem}\label{t.FFM} There exists an origami $\widetilde{L}$ of genus $11$ such that the restriction of the KZ cocycle over $SL(2,\mathbb{R})\cdot \widetilde{L}$ to a certain  $SL(2,\mathbb{R})$-irreducible piece of the corresponding semisimple decomposition acts (modulo finite-index subgroups) through a Zariski dense subgroup of $SO^*(6)$ in its standard representation.
\end{theorem}

\begin{remark}
This result says that a particular case of item (iv) above occurs for the KZ cocycle over the $SL(2,\mathbb{R})$-action on moduli spaces of translation surfaces.
In fact, we have an exceptional isomorphism of real Lie algebras $\mathfrak{so}^*(6)\cong \mathfrak{su}_{3,1}$.
This can be seen, for instance, by comparing the Satake diagrams in the back of \cite{Vinberg}.

Moreover, letting $\mathfrak{so}^*(6)$ act in its standard representation identifies it with $\mathfrak{su}_{3,1}$ acting in the second exterior power of its standard representation.
The underlying vector space is $\mathbb{C}^6$ viewed as $\mathbb{R}^{12}$.

In particular, up to finite center, the groups $SO^*(6)$ and $SU(3,1)$ are isomorphic.
This means that the example from \autoref{t.FFM} is, in fact, also an example of $SU(3,1)$ acting in its second exterior power (thus, of item (iii) in the list above).
\end{remark}

\begin{remark}


In principle, the computations from \autoref{sec:multiplicities} as well as the discussion in \autoref{sec:monodromy} suggest how to find further examples of pieces of the monodromy in $SO^*(2n)$ for any $n$.
Namely, one can first look for quarternionic representations occurring in the cohomology of some regular origami (or more general family of translation surfaces with symmetries).
The multiplicity $n$ of the representation will constrain the monodromy to lie inside $SO^*(2n)$.
One then has to check that the monodromy is irreducible, i.e. that the cocycle does not split further.
Note that the computations of Matheus, Yoccoz, and Zmiaikou \cite{MYZ} give rise to a large number of explicit examples.

On the other hand, the question of finding examples of $SL(2,\mathbb{R})$-orbits in moduli spaces of translation surfaces whose associated KZ cocycles have monodromy groups with non-compact factors realizing all cases in items (iii) or (v) seems more challenging in our opinion (cf. \autoref{foot.2}).
\end{remark}

Regarding the multiplicities of representations for the cohomology of regular origamis, we have the next result.
\begin{theorem}
 Let $S$ be a regular origami, determined by a group $G$ and the two generator $h,v\in G$ (see \autoref{s.preliminaries} for definitions).
 Let $c:=hvh^{-1}v^{-1}$ be their commutator, and let $\pi$ be a complex irreducible representation of $G$.
 Then in the cohomology group $H^1(S;\mathbb{C})$, the representation $\pi$ appears (see \autoref{eq:top_mult}) with multiplicity
 \begin{align*}
 2 \delta_{\pi=triv} + \dim \pi - m_0
 \end{align*}
 Here $\delta_{\pi=triv}$ is a constant equal to $1$ if $\pi$ is trivial and zero otherwise, while $m_0$ denotes the dimension of the vector subspace in $\pi$ fixed by the element $c$.
 
 In the cohomology group $H^{0,1}(S)$, the representation $\pi$ appears (see \autoref{eq:holom_mult}) with multiplicity
 $$
 \delta_{\pi=triv} + \frac 12 (\dim \pi - m_0) + \frac 1N \left( \sum_{i=1}^{N-1} \left(i-\frac N2\right)m_i \right)
 $$
 Here $m_i$ denotes the dimension of the space in the representation $\pi$ on which the element $c$ acts with eigenvalue $\exp(2\pi \sqrt{-1} \frac i N)$, while $N$ is the order of $c$ in $G$.
\end{theorem}

For regular origamis, the previous theorem answers a question posed in Remark 5.13 of \cite{MYZ}. 

\paragraph{Paper outline.}
In \autoref{s.preliminaries} we recall some basic facts about origamis.
Next, in \autoref{sec:multiplicities} we explicitly compute the multiplicities of representations for regular origamis.
This extends the computations in \cite{MYZ} to the multiplicities in the $(1,0)$ and $(0,1)$ parts of the Hodge decomposition.
In \autoref{sec:monodromy}, we discuss the general structure of (semisimple) local systems with symmetries.
We also include a more detailed description of the group $SO^*(2n)$.

The construction of the example and the proof of \autoref{t.FFM} occupy the remainder of the note.
We introduce in \autoref{s.quaternion-cover} the origami $\widetilde{L}$ (the main object of this note), and we determine the Lyapunov spectrum of the KZ cocycle over $SL(2,\mathbb{R})\cdot\widetilde{L}$.
In particular, the structure of the Lyapunov spectrum allows us to show that the monodromy group of the KZ cocycle over
$SL(2,\mathbb{R})\cdot\widetilde{L}$ has a non-compact factor isomorphic to either $SO^*(4)$ or $SO^*(6)$ (modulo compact and finite-index subfactors), cf. \autoref{t.FFM1} below.
Finally, in \autoref{s.irreducibility-exotic-monodromy} we complete the proof of \autoref{t.FFM} by ruling out the possibility in \autoref{t.FFM1} of a $SO^*(4)$ monodromy through the computation of certain matrices\footnote{I.e., the actions on the homology group $H_1(\widetilde{L},\mathbb{R})$ of certain affine homeomorphisms of $\widetilde{L}$.} of the KZ cocycle over $SL(2,\mathbb{R})\cdot\widetilde{L}$.

\section{Preliminaries}\label{s.preliminaries}

In this section, we recall some useful facts about origamis (square-tiled surfaces) and the $SL(2,\mathbb{R})$-action on the moduli spaces of translation surfaces. For this sake, we will loosely follow the exposition in \cite{MYZ} for a large portion of this section.

\subsection{Origamis}
\label{subsec:origamis}
An \emph{origami} (or \emph{square-tiled surface}) is a translation surface $X=(M,\omega)$ such that the Riemann surface $M$ is obtained by a finite covering $\pi:M\to\mathbb{T}^2$ of the torus $\mathbb{T}^2=\mathbb{R}^2/\mathbb{Z}^2$ which is unramified off $0\in\mathbb{T}^2$, and the (non-zero) Abelian differential $\omega$ on $M$ is the pullback $\omega=\pi^*(dz)$ of the Abelian differential $dz$ on $\mathbb{C}/(\mathbb{Z}\oplus i\mathbb{Z})\simeq \mathbb{T}^2$.

Alternatively, an origami is a translation surface $X=(M,\omega)$ is determined by a pair of permutations $h,v\in S_N$ of the set $\{1,\dots, N\}$ through the following recipe. We take $N$ copies $sq_1,\dots, sq_N$ of the unit square $[0,1]^2\subset\mathbb{R}^2$, and, for each $n\in\{1,\dots, N\}$, we glue by translation the rightmost vertical side of $sq_n$ to the leftmost vertical side of $sq_{h(n)}$, resp. the topmost horizontal side of $sq_n$ to the bottommost horizontal side of $sq_{v(n)}$. In this way, after performing these identifications, we obtain a Riemann surface $M$ that is naturally equipped with an Abelian differential $\omega$ given by the pullback of $dz$ on each $sq_n$ (observe that this makes sense because the identifications are given by translations on $\mathbb{R}^2$). Note that the translation surface $X=(M,\omega)$ associated to a pair of permutations $h,v\in S_N$ is connected if and only if the group generated by $h$ and $v$ acts transitively on $\{1,\dots, N\}$.

These two definitions of origamis are related as follows. On one hand, given a translation surface $X=(M,\omega)$ associated to a pair of permutations $h,v\in S_N$, the natural projection from each $S_n$, $n=1,\dots, N$, to $\mathbb{T}^2$ induces a finite covering $\pi:M\to \mathbb{T}^2$ which is unramified off $0\in\mathbb{T}^2$ such that $\omega=\pi^*(dz)$. On the other hand, given a finite covering $\pi:M\to\mathbb{T}^2$ of degree $N$ unramified off $0\in\mathbb{T}^2$, the closures in $M$ of each  of the $N$ connected components of $\pi^{-1}((0,1)^2)$ determine $N$ copies of the unit square $[0,1]^2$. After numbering $sq_1,\dots, sq_N$ the connected components of $\pi^{-1}((0,1)^2)$ in some arbitrary way, we obtain a pair of permutations $h,v\in S_N$ such that, for each $n=1,\dots, N$, the neighbor to the right, resp. on the top, of $sq_{n}$ is $sq_{h(n)}$, resp. $sq_{v(n)}$.

\begin{remark}\label{r.simultaneous-conjugation} Observe that the particular choice of numbering of the connected components of $\pi^{-1}((0,1)^2)$ is not relevant from the point of view of translation surfaces: in other terms, by replacing a pair of permutations $h, v\in S_N$ by another pair $\phi h \phi^{-1}, \phi v\phi^{-1}\in S_N$ of permutations that are simultaneously conjugated to $h,v$, we obtain the same origami $X=(M,\omega)$.
\end{remark}

\subsection{\texorpdfstring{$SL(2,\mathbb{R})$}{Sl2}-action on strata of moduli spaces of translation surfaces}
Let $X = (M,\omega)$ be a translation surface of genus $g\geq1$. Since $\omega$ is a non-zero Abelian differential (by definition), we have that the set $\Sigma$ of zeroes of $\omega$ is finite, and, by Riemann-Roch theorem, the list $\kappa=(k_1,\dots,k_{\sigma})$, $\sigma=\#\Sigma$, of orders of zeroes of $\omega$ satisfy
$$\sum\limits_{l=1}^{\sigma} k_l = 2g-2.$$

Given a finite subset $\Sigma\subset M$ and a list $\kappa=(k_1,\dots,k_{\sigma})$, $\sigma=\#\Sigma$, of natural numbers satisfying $\sum\limits_{l=1}^{\sigma} k_l = 2g-2$, the corresponding \emph{stratum} $\mathcal{H}(\kappa)$, resp. $\mathcal{T}(\kappa)$, of the \emph{moduli space}, resp. \emph{Teichm\"uller space, of translation surfaces} of genus $g$ is the space of orbits of translation surfaces $X=(M,\omega)$ such that $\omega|\Sigma=0$ and the list of orders of zeroes of $\omega$ is $\kappa$ under the natural action of the group $\textrm{Homeo}^+(M,\Sigma,\kappa)$, resp. $\textrm{Homeo}_0(M,\Sigma,\kappa)$, of orientation-preserving homeomorphisms of $M$ that preserve $\Sigma$ and $\kappa$, resp. the connected component in $\textrm{Homeo}^+(M,\Sigma,\kappa)$ of the identity element. Note that $\mathcal{H}(\kappa)=\mathcal{T}(\kappa)/\Gamma(M,\Sigma,\kappa)$ where $\Gamma(M,\Sigma,\kappa):=\textrm{Homeo}^+(M,\Sigma,\kappa) / \textrm{Homeo}_0(M,\Sigma,\kappa)$ is the so-called \emph{mapping class group}
 of isotopy classes of orientation-preserving homeomorphisms of $M$ respecting $\Sigma$ and $\kappa$.

The group $SL(2,\mathbb{R})$ acts on $\mathcal{H}(\kappa)$ as follows. A translation surface $X=(M,\omega)\in \mathcal{H}(\kappa)$ is determined by the so-called \emph{translation charts} provided by the local primitives of $\omega$ on $M-\Sigma$. Given $g\in SL(2,\mathbb{R})$ and $X=(M,\omega)\in\mathcal{H}(\kappa)$, we define $g(X)$ as the translation surface obtained by post-composition of the translation charts of $X=(M,\omega)$ with $g$. In this setting, the action of the diagonal subgroup $g_t=\textrm{diag}(e^t, e^{-t})$ of $SL(2,\mathbb{R})$ on strata of moduli spaces of translation surfaces is called \emph{Teichm\"uller geodesic flow}.

\subsection{Veech groups and affine homeomorphisms}

The $SL(2,\mathbb{R})$-orbits of ori\-gamis are particular examples of closed $SL(2,\mathbb{R})$-orbits in the strata of the moduli spaces of translation surfaces.

The geometry of the $SL(2,\mathbb{R})$-orbit of an origami $(M,\omega)$ admits the following simple description. Let $SL(M,\omega)$ be the \emph{Veech group} of $(M,\omega)$, i.e., the stabilizer of $(M,\omega)$ with respect to the $SL(2,\mathbb{R})$-action. The Veech group $SL(M,\omega)$ of an origami $(M,\omega)$ is commensurable to $SL(2,\mathbb{Z})$. In particular, if $(M,\omega)$ is an origami, the hyperbolic surface $\mathbb{H}/SL(M,\omega)$ has finite area. As it turns out, the $SL(2,\mathbb{R})$-orbit $SL(2,\mathbb{R})\cdot (M,\omega)$ of an origami $(M,\omega)$ is naturally isomorphic to the unit cotangent bundle $SL(2,\mathbb{R})/SL(M,\omega)$ of a finite-area hyperbolic surface.

The Veech group $SL(M,\omega)$ relates to the flat geometry of translation surfaces via the notion of \emph{affine homeomorphisms}. The group $\textrm{Aff}(M,\omega)$ of affine homeomorphisms of $X=(M,\omega)$ consists of all orientation-preserving homeomorphisms of $M$ (respecting $\Sigma$) whose expressions in translation charts (of local primitives of $\omega$ on $M-\Sigma$) are affine transformations of $\mathbb{R}^2$. By extracting the linear part of these affine transformations, we obtain a homomorphism from $\textrm{Aff}(X)$ to $SL(2,\mathbb{R})$ whose kernel $\textrm{Aut}(X)$ is the so-called \emph{automorphism group} of $X=(M,\omega)$. Moreover, this homomorphism fits into an exact sequence
$$\{\textrm{Id}\}\to \textrm{Aut}(X)\to \textrm{Aff}(X)\to SL(X)\to\{\textrm{Id}\}$$
In particular, we have that the Veech group $SL(X)$ is the subgroup of $SL(2,\mathbb{R})$ capturing all linear parts of affine homeomorphisms of $X$.

The group $\textrm{Aff}(X)$ of affine homeomorphisms is the stabilizer in the mapping class group $\Gamma(M,\Sigma,\kappa)$ of the $SL(2,\mathbb{R})$-orbit of
$X=(M,\omega)$.

\subsection{KZ cocycle over the \texorpdfstring{$SL(2,\mathbb{R})$}{SL2R}-orbit of an origami} The \emph{Kontsevich-Zorich cocycle} over the $SL(2,\mathbb{R})$-orbit of a translation surface $X=(M,\omega)$ is the quotient of the trivial cocycle
$$SL(2,\mathbb{R})\cdot X \times H_1(M,\mathbb{R})\to SL(2,\mathbb{R})\cdot X \times H_1(M,\mathbb{R})$$
(over the tautological dynamics of $SL(2,\mathbb{R})$ on $SL(2,\mathbb{R})\cdot X$) by the natural action (on both factors of the trivial cocycle) of the stabilizer $\textrm{Aff}(X)$ in the mapping class group of the $SL(2,\mathbb{R})$-orbit of $X$ in a stratum of the moduli space of translation surfaces. Equivalently, the KZ cocycle over $SL(2,\mathbb{R})\cdot X\simeq SL(2,\mathbb{R})/SL(X)$ acts on $H_1(M,\mathbb{R})$ via appropriate elements of $\textrm{Aff}(X)$.

Suppose from now on that $X=(M,\omega)$ is an origami and let $\pi:M\to\mathbb{T}^2$ be the associated finite ramified covering. In this case, the KZ cocycle respects the following decomposition into ($\textrm{Aff}(X)$-invariant) subspaces defined over $\mathbb{Q}$:
$$H_1(M,\mathbb{R}) = H_1^{st}(M,\mathbb{R})\oplus H_1^{(0)}(M,\mathbb{R}),$$
where $H_1^{(0)}(M,\mathbb{R})$ is the kernel of $\pi_*: H_1(M,\mathbb{R})\to H_1(\mathbb{T}^2,\mathbb{R})$, and $H_1^{st}(M,\mathbb{R})$ is the symplectic orthogonal of $H_1^{(0)}(M,\mathbb{R})$ with respect to the usual (symplectic) intersection form.

The subspace $H_1^{st}(M,\mathbb{R})$ is naturally identified with $\mathbb{R}^2\simeq H_1(\mathbb{T}^2,\mathbb{R})$ because it is generated by the absolute homology classes $\sigma=\sum\limits_s \sigma_s$ and $\zeta=\sum\limits_s \zeta_s$ where $s$ runs through the set $Sq(X)$ of squares of $X$ (i.e., the closures of connected components of $\pi^{-1}((0,1)^2)$), and $\sigma_s$, resp. $\zeta_s$, is the bottommost, resp, leftmost, horizontal, resp. vertical, side of $s$. Furthermore, under this identification,
the affine group $\textrm{Aff}(X)$ acts on $H_1^{st}(M,\mathbb{R})\simeq\mathbb{R}^2$ via the composition of the homomorphism $\textrm{Aff}(X)\to SL(X)$ and the standard action of $SL(X)\subset SL(2,\mathbb{R})$ on $\mathbb{R}^2$. In other terms, the monodromy group of the restriction of the KZ cocycle to $H_1^{st}(M,\mathbb{R})$ is $SL(2,\mathbb{R})$ (up to finite-index).

Denoting by $G = \textrm{Aut}(X)$ the (finite) group of automorphisms of the origami $X=(M,\omega)$, we have that the subspaces $H_1^{st}(M,\mathbb{R})$ and $H_1^{(0)}(M,\mathbb{R})$ are $G$-modules. In particular, we can write
$$H_1^{(0)}(M,\mathbb{R}) = \bigoplus\limits_{a\in \textrm{Irr}_{\mathbb{R}}(G)} W_a$$
where $\textrm{Irr}_{\mathbb{R}}(G)$ is the set of (isomorphism classes of) irreducible representations of $G$, and $W_a$ is the isotypical component of $a\in\textrm{Irr}_{\mathbb{R}}(G)$ in the $G$-module $H_1^{(0)}(M,\mathbb{R})$.

The isotypical components $W_a$, $a\in\textrm{Irr}_{\mathbb{R}}(G)$, of $H_1^{(0)}(M,\mathbb{R})$ are permuted by the group $\textrm{Aff}(X)$ of affine homeomorphisms of the origami $X=(M,\omega)$. More precisely, $\textrm{Aff}(X)$ acts on $G=\textrm{Aut}(X)$ by conjugation, that is, we have a natural homomorphism $\textrm{Aff}(X)\to\textrm{Aut}(G)$. Next, we recall that one has a canonical homomorphism $\textrm{Aut}(G)\to \textrm{Out}(G):= \textrm{Aut}(G)/\textrm{Inn}(G)$ where $\textrm{Inn}(G)$ is the (normal) subgroup of inner automorphisms of $G$. Furthermore, $\textrm{Out}(G)$ acts on $\textrm{Irr}_{\mathbb{R}}(G)$. In this context, the composition $\textrm{Aff}(X)\to \textrm{Aut}(G)\to \textrm{Out}(G)$ of the two homomorphisms described above induces an action of $\textrm{Aff}(X)$ on $\textrm{Irr}_{\mathbb{R}}(G)$ such that the elements of $\textrm{Aff}(X)$ permute the isotypical components $W_a$ according to this action (i.e., $A(W_a)=W_{A\cdot a}$ for every $A\in\textrm{Aff}(X)$ and $a\in\textrm{Irr}_
{\mathbb{R}}(G)$).

Denote by $\textrm{Aff}_{**}(X)$ the kernel of the homomorphism $\textrm{Aff}(X)\to\textrm{Aut}(G)$. By definition, $\textrm{Aff}_{**}(X)$ is a finite-index subgroup of $\textrm{Aff}(X)$ such that the elements of $\textrm{Aff}_{**}(X)$ fix the isotypical components $W_a$ of $H_1^{(0)}(M,\mathbb{R})$, and, moreover, the restriction of these elements to each $W_a$ are automorphisms of $G$-module (because the elements of $\textrm{Aff}_{**}(X)$ commute with $G=\textrm{Aut}(X)$). Since the isotypical components $W_a$ are mutually orthogonal with respect to the symplectic intersection form on $H_1(M,\mathbb{R})$ and the restriction to each $W_a$ of the intersection form is also symplectic form, we deduce that the following restriction on the monodromy group of the KZ cocycle: the elements of $\textrm{Aff}_{**}(X)$ act via a subgroup of the product of the groups $Sp(W_a)$ of automorphisms of the $G$-modules $W_a$ preserving the symplectic intersection form.

\begin{remark}\label{r.semisimplicity} In principle, the action of $\textrm{Aff}_{**}(X)$ in a given isotypical component $W_a$ might not be irreducible. In this case, M\"oller \cite{Moller} (see also \cite{Filip1} for a more general version) showed the following ``Deligne's semisimplicity theorem'': $W_a$ can be further decomposed into $\textrm{Aff}_{***}(X)$-irreducible  symplectic subspaces (respecting the so-called Hodge structure of $H^1(M,\mathbb{R})$), where $\textrm{Aff}_{***}(X)$ is an appropriate finite-index subgroup of $\textrm{Aff}_{**}(X)$.
\end{remark}

\begin{definition}
  The \emph{monodromy} of the Kontsevich-Zorich cocycle is the image of the group $\textrm{Aff}(X)$ inside the product of the symplectic groups $Sp(W_a)$.
  Since the groups $\textrm{Aff}_{**}(X)$ and $\textrm{Aff}_{***}(X)$ are finite-index in $\textrm{Aff}(X)$, they map with finite index to the monodromy.
  
  Therefore, the connected component of the Zariski closure of the monodromy is independent of which group in the family $\textrm{Aff}_\bullet(X)$ we consider for the action.
\end{definition}

The long-term behavior of KZ cocycle (action of $\textrm{Aff}(X)$ on homology) is described by its \emph{Lyapunov spectrum}. More precisely, the $SL(2,\mathbb{R})$-orbit of an origami $X=(M,\omega)$ is isomorphic to $SL(2,\mathbb{R})/SL(X)$ where $SL(X)$ is commensurable to $SL(2,\mathbb{Z})$. Thus, $SL(2,\mathbb{R})\cdot X$ carries an unique $SL(2,\mathbb{R})$-invariant ergodic probability measure $\mu$. In this setting, Oseledets theorem says that, for $\mu$-almost every $x\in SL(2,\mathbb{R})\cdot X$, the sequence $(A_n)_{n\in\mathbb{Z}}\subset \textrm{Aff}_{**}(X)$ of matrices of the KZ cocycle (depending on $x$) along the orbit $g_t\cdot x$ of $x$ under the Teichm\"uller geodesic flow $g_t = \textrm{diag}(e^{t}, e^{-t})$ has the following asymptotic behavior: for all $v\in H_1(M,\mathbb{R})-\{0\}$, we have well-defined quantities
$$\lim\limits_{n\to\pm\infty} \frac{\log\|A_n(v)\|}{\log\|A_n\|} = \theta$$
independently of the choice of the norm $\|.\|$ on $H_1(M,\mathbb{R})$; furthermore, the collection of numbers $\theta$ obtained in this way is a finite list of numbers (with multiplicities) $\{\theta_1\geq\theta_2\geq\dots\geq\theta_{2g}\}$ of cardinality $2g=\textrm{dim}_{\mathbb{R}}H_1(M,\mathbb{R})$ which is independent of $x$. In the literature, the list $\{\theta_1\geq\dots\geq\theta_{2g}\}$ is called Lyapunov spectrum, the quantities $\theta_{\alpha}$ are called \emph{Lyapunov exponents}, and
$$E(\theta_{\alpha}, x):=\left\{v\in H_1(M,\mathbb{R})-\{0\}: \lim\limits_{n\to\pm\infty} \frac{\log\|A_n(v)\|}{\log\|A_n\|} = \theta_{\alpha} \right\}$$
are called \emph{Oseledets subspaces}\footnote{These subspaces depend measurably on $x$.}.

The Lyapunov spectrum of the KZ cocycle is symmetric with respect to the origin, i.e., $\theta_{2g-\alpha+1}=\theta_{\alpha}$ for all $\alpha=1,\dots, g$. Indeed, this is a consequence of the fact that the KZ cocycle is symplectic (the action of $\textrm{Aff}(X)$ on $H_1(M,\mathbb{R})$ preserves the symplectic intersection form). In particular, the Lyapunov spectrum of the KZ cocycle is always determined by its non-negative Lyapunov exponents $\theta_1\geq \dots\geq \theta_g(\geq 0)$.

It is known that the top Lyapunov exponent $\theta_1$ equals $1$, it is simple (i.e., $\theta_1>\theta_2$), and it comes from the action of the Teichm\"uller geodesic flow $g_t$ on $H_1^{st}(M,\mathbb{R})$.

On the other hand, the remaining Lyapunov exponents might exhibit multiplicities depending on the ``symmetries'' of $X=(M,\omega)$. For example, for any $A\in \textrm{Aff}(X)$, the Lyapunov spectrum of the (restriction of the KZ cocycle to the) isotypical components $W_a$ and $W_{A\cdot a}$, $a\in\textrm{Irr}_{\mathbb{R}}(G)$, are the same. Moreover, for each isotypical component $W_a$, the associated Oseledets subspaces $W_a(\theta, x) = W_a\cap E(\theta, x)$ are $G$-modules. In particular, if the representation $a\in\textrm{Irr}_{\mathbb{R}}(G)$ is complex, resp. quaternionic, then each Lyapunov exponent of $W_a$ has multiplicity at least $2$, resp. $4$.

Closing this section, we recall that Eskin-Kontsevich-Zorich \cite{EKZ1} proved the following explicit formula for the sum of the non-negative Lyapunov exponents of the KZ cocycle over the $SL(2,\mathbb{R})$-orbit of an origami. Let $X=(M,\omega)\in\mathcal{H}(\kappa)$, $\kappa=(k_1, \dots, k_{\sigma})$, be an origami of genus $g$. Consider the $SL(2,\mathbb{Z})$-orbit of $X$: this is a finite set\footnote{The cardinality of $SL(2,\mathbb{Z})\cdot X$ is the index of the Veech group $SL(X)$ in $SL(2,\mathbb{Z})$ when the origami $X=(M,\omega)$ is reduced (i.e., the covering $\pi:M\to\mathbb{T}^2$ such that $\pi^*(dz)=\omega$ does not factor through a cover $p:\mathbb{T}^2\to\mathbb{T}^2$ of degree $\textrm{deg}(p)>1$).} of origamis that one can compute by successively applying the generators $T = \left(\begin{array}{cc} 1 & 1 \\ 0 & 1 \end{array}\right)$ and
$S = \left(\begin{array}{cc} 1 & 0 \\ 1 & 1 \end{array}\right)$ of $SL(2,\mathbb{Z})$ to the origamis. In terms of pairs of permutations $h, v\in S_N$ associated to origamis, $T$ and $S$ acts as $T(h,v)=(h,vh^{-1})$ and $S(h,v)=(hv^{-1}, v)$, and this allows\footnote{Here, one has to keep in mind that two pairs of permutations give the same origami when they are simultaneously conjugated, that is, $(h,v)$ and $(\phi h \phi^{-1}, \phi v\phi^{-1})$ correspond to the same origami.} to calculate the $SL(2,\mathbb{Z})$-orbits of origamis. In this context, the sum of the non-negative Lyapunov exponents $1=\theta_1\geq\dots\geq\theta_g$ of the KZ cocycle over $SL(2,\mathbb{R})\cdot X$ is
$$\theta_1+\dots+\theta_g = \frac{1}{12}\sum\limits_{l=1}^{\sigma} \frac{k_l(k_l+2)}{k_l+1} + \frac{1}{\# SL(2,\mathbb{Z})\cdot X} \sum\limits_{\substack{Y\in SL(2,\mathbb{Z})\cdot X,
\\ c \textrm{ is a cycle of } h_Y}} \frac{1}{\textrm{length of } c}$$
where $(h_Y, v_Y)$ is a pair of permutations associated to the origami $Y$.

\section{Multiplicities of representations for regular origamis}
\label{sec:multiplicities}
This section calculates the multiplicities of representations that occur in the cohomology of regular origamis.
For the topological case, these calculations were done by Matheus, Yoccoz, and Zmiaikou \cite{MYZ}.
We also compute the multiplicities for the holomorphic $1$-forms (cf. \cite[Remark 5.13]{MYZ}).

Two tools appear in the computation.
One is the Lefschetz fixed point formula and its holomorphic version.
The second is the Frobenius reciprocity formula - it reduces the calculation to a cyclic subgroup of the origami symmetries.

\paragraph{Setup.} Calculations will take the algebraic point of view on origamis (cf. \autoref{subsec:origamis}).
Let $G$ be a finite group generated by two elements $h, v\in G$ and let their commutator be $c:=[h,v]=hv h^{-1} v^{-1}$.
Let $S$ be the corresponding regular origami.
The squares of $S$ correspond to the elements of $G$, and there are two edges in the glued surface per element of $G$.

On the origami, the ramification points (i.e. vertices of the square) are in natural bijection with the coset space $G/\ip{c}$.
Here $\ip{c}$ denotes the cyclic subgroup of $G$ generated by $c$.
Denoting its cardinality by $N:=|\ip{c}|$ the genus of the surface satisfies
$$
\textrm{genus}(S) = \frac 12 \left(|G|-\frac{|G|} {N}\right) +1
$$

\subsection{Preliminaries from representation theory}
\label{subsec:rep_th_prelims}
We now recall some necessary facts from the representation theory of finite groups.
This material is available in Serre's textbook \cite[Ch. 7]{Serre}.

Consider two finite groups $H\subset G$.
For two complex-valued functions $\phi,\psi$ on $G$, define
$$
\ip{\phi,\psi}_G = \frac 1 {|G|} \sum_{g\in G} \phi(g)\overline{\psi(g)}
$$
A similar definition applies to $\ip{,}_H$.

For $G$-representations $V$ and $\pi$ (with $\pi$ irreducible) denoting their characters by $\chi_V$ and $\chi_\pi$ we have
$$
\ip{\chi_\pi,\chi_V}_G = \dim \Hom_G(\pi, V)
$$
In other words, the product of characters gives the multiplicity of $\pi$ in $V$.

\paragraph{Restriction.}
Given a representation $\pi$ of $G$, let $\Res^G_H(\pi)$ denote the representation of $H$ obtained by restricting $\pi$ to it.
If the character of $\pi$ is $\chi_\pi$, let $\Res^G_H \chi_\pi$ denote the character of the restricted representation.
We will often omit the notation of $G$ and $H$ from $\Res$, since we will work with only one such pair.
Note that $\Res\chi_\pi$ is simply the restriction of the function $\chi_\pi$ on $G$ to the subset (and subgroup) $H$.

\paragraph{Induction.}
Given a representation $\pi$ of $H$, let $\Ind^G_H(\pi)$ denote the representation of $G$ induced from $\pi$.
If the character of $\pi$ is $\chi_\pi$, then the character of the induced representation is given by the formula
\begin{align}
\label{eq:ind_rep}
\Ind^G_H \chi_\pi(g) = \frac 1 {|H|} \sum_{\substack{s\in G\\ sgs^{-1}\in H} } \chi_\pi (sgs^{-1})
\end{align}
The factor of $1/|H|$ accounts for the action of $H$ on the set of $s\in G$ such that $sgs^{-1}\in H$.

\paragraph{Frobenius reciprocity.}
The relation between restriction and induction is given by the Frobenius reciprocity formula:
$$
\ip{\Ind^G_H \chi_{\pi_1}, \chi_{\pi_2} } = \ip{ \chi_{\pi_1}, \Res^G_H \chi_{\pi_2} }_H
$$
Here, $\pi_1$ is a representation of $H$ and $\pi_2$ a representation of $G$.

We shall apply this formula to the group $G$ coming from the regular origami $S$ and to its cyclic subgroup $H=\ip{c}$.
These will be omitted from the notation in $\Res$ and $\Ind$.
Finally, observe that inducing the trivial representation of $H$ to $G$ gives the representation of $G$ on functions on $G/H$.

\paragraph{Real, Complex, Quaternionic representations.}
Let $\pi$ be an irreducible representation of $G$ on a complex vector space.
We have the following possibilities for $\pi$:
\begin{enumerate}
 \item Real representation. The character is real-valued and the representation is induced from the complexification of a real representation.
 \item Complex representation. The character is complex-valued. The representation $\pi\oplus \overline{\pi}$ is induced from a representation on a real vector space.
 \item Quaternionic representation. The character is real-valued, but the representation is \emph{not} induced from the complexification of some real representation.
 The representation $\pi\oplus \pi$, however, is induced from a representation on a real vector space.
\end{enumerate}

Let now $\pi$ be an irreducible representation of $G$ on a real vector space.
Unlike the case of complex vector spaces above, $\pi$ can have non-trivial endomorphisms.
We have the following possibilities:
\begin{enumerate}
 \item Real representation. The endomorphisms equal $\mathbb{R}$ and after complexification, $\pi$ is still irreducible.
 \item Complex representation. The endomorphisms equal $\mathbb{C}$ and after complexification, $\pi$ becomes isomorphic to $\pi_1\oplus \overline{\pi_1}$, where $\pi_1$ is a complex representation in the sense above.
 \item Quaternionic representation. The endomorphisms equal $\mathbb{H}$ and after complexification, $\pi$ becomes isomorphic to $\pi_1\oplus \pi_1$, where $\pi_1$ is a quaternionic representation in the sense above.
\end{enumerate}

\subsection{Isotypical components in the topological cohomology}
In this section, we shall compute the multiplicity of an irreducible representation of $G$ in the first cohomology of the origami $S$.

\paragraph{Lefschetz fixed point formula.}
Recall the following statement, described for instance in the monograph of Griffiths and Harris \cite[p. 421]{GH}.
Let $g$ be a smooth diffeomorphism with isolated fixed points of a compact manifold $S$.
Then we have the following formula for the alternating sum of traces for the action on cohomology:
$$
\sum_{i=0}^{\dim S} (-1)^i \tr (\left. g^* \middle|_{H^i(S)} \right. ) = \sum_{g(p)=p} \textrm{index}(p)
$$

\paragraph{The case of origamis.}
In our case, $S$ is a smooth surface and all fixed points have index $1$ (since the diffeomorphisms respect a complex structure).
We have a full action of the group $G$ and let $\chi_{H^1}$ be the character of the $G$-representation of the first cohomology group $H^1(S)$ (coefficients in $\mathbb{C}$).

Applying the Lefschetz fixed point formula to all $g$ not the identity, we find
$$
\chi_{H^1}(g) = 
\begin{cases}
 2\cdot \textrm{genus}(S) & \textrm{ if } g=id\\
 2-\#\textrm{fixed pts.}(g) & \textrm{ otherwise}
\end{cases}
$$
\paragraph{Counting fixed points.}
Assume $g\in G$ is not the identity.
Then the fixed points of the action of $g$ on $S$ can only be among the vertices of the squares tiling $S$.
Moreover, the action of $g$ on them is the same as the action on the coset $G/\ip{c}$, which is naturally identified with the set of vertices.
Letting $\chi_0$ denote the character of the trivial representation of $\ip{c}$, the number of fixed points of $g$ acting on $S$ is therefore given by $\Ind^{G}_{\ip{c}} \chi_0 (g)$.

Therefore, the character of the $G$-representation on $H^1(S)$ can be written as
\begin{align}
\label{eq:chi_H1}
\chi_{H^1} = 2 - \Ind \chi_0 + \left(\Ind {\chi_0} (id)-2 + 2 \textrm{genus}(S) \right)\delta_{id}
\end{align}
Here $\delta_{id}$ is the delta-function at identity on $G$.
Note also that $\Ind \chi_0 (id) = \dim (\Ind \chi_0) = |G|/N$.

\paragraph{Computing multiplicities.}
The multiplicity of an irreducible representation $\pi$ of $G$ in $\chi_{H^1}$ is given by $\ip{\chi_{H^1},\chi_\pi}_G$.
Using Frobenius reciprocity, this will reduce to a computation on the cyclic group generated by $c$.

We have the following identities for inner products of functions on $G$ (where $const.$ denotes a constant function):
\begin{align*} 
 \ip{\delta_{id},\chi_\pi}_G &= \frac{\dim \pi}{|G|}\\
 \ip{const.,\chi_\pi}_G &=
  \begin{cases}
    const.	& \textrm{ if } \pi \textrm{ is trivial}\\
    0 		& \textrm{ otherwise}
  \end{cases}
\end{align*}
Taking the product of $\chi_\pi$ with $\chi_{H^1}$, using \autoref{eq:chi_H1} and the formulas above, we find
\begin{align*}
 \ip{\chi_{H^1},\chi_\pi} & = 2\delta_{\pi=triv} - \ip{\Ind \chi_0, \chi_\pi}_G + \frac{\dim \pi}{|G|}\left( \frac{|G|}{N} -2 + \left(|G| - \frac{|G|}N + 2\right) \right)\\
 &= 2 \delta_{\pi=triv} + \dim \pi - \ip{\Ind \chi_{0},\chi_\pi}_G
\end{align*}
Above and below, $\delta_{\pi=triv}$ denotes a constant which is $1$ if $\pi$ is the trivial representation, and $0$ otherwise.
Using Frobenius reciprocity from $G$ to the subgroup $\ip{c}$ to evaluate the last term, we find
\begin{align}
\label{eq:top_mult}
\ip{\chi_{H^1},\chi_\pi} = 2 \delta_{\pi=triv} + \dim \pi - \ip{\chi_0,\Res \chi_\pi}_{\ip{c}}
\end{align}
The term $\ip{\chi_0,\Res \chi_\pi}_{\ip{c}}$ counts the multiplicity of the trivial representation of the restriction of $\pi$ to the cyclic group $\ip{c}$.
This is exactly the dimension of the space of vectors fixed by $c\in G$ in the representation $\pi$. 
Note that \autoref{eq:top_mult} recovers a result of Matheus-Yoccoz-Zmiaikou \cite[Cor. 3.5]{MYZ}.

\subsection{Isotypical components in the holomorphic forms}
The pattern for the calculation is similar to the one above.
The Lefschetz fixed point formula has a holomorphic version, which is a bit more involved.
To have a more convenient algebraic framework to analyze it, we introduce an extra parameter and expand our functions in power series.

\paragraph{Holomorphic Lefschetz fixed point formula.}
Recall the following statement, described for instance in the monograph of Griffiths and Harris \cite[p. 426]{GH}.
Let $g$ be a holomorphic self-map of a complex manifold $S$, with isolated fixed points.
Then the alternating sum of traces on the Dolbeault cohomology groups is given by
$$
\sum_{i=0}^{\dim_\mathbb{C} S} (-1)^i \tr \left( g^* \middle|_{H^{0,i}}(S) \right) = 
\sum_{ \substack{g(p)=p\\ (Dg)_p:T_p S\to T_p S } } \frac {1} {\det\left(1-(Dg)_p\right)}
$$
In our situation, $S$ is a compact Riemann surface so the formula becomes
$$
1-\tr \left(g^*\middle|_{H^{0,1}(S)} \right) = \sum_{g\cdot p=p} \frac {1}{1-\mu_p(g)}
$$
where $\mu_p(g)$ is the derivative of $g$ at its fixed point $p$ (thus, a complex scalar).
Note that $\mu_p(g)$ is never $1$ when $g$ is not the identity.

\paragraph{Derivative at fixed points for origamis.}
In our situation, the fixed points are among the elements of the coset $G/\ip{c}$.
Corresponding to one such $p\in S$ we have a coset $h_p\ip{c}$.
This means that $g\cdot h_p \cdot \ip{c} = h_p\cdot \ip{c}$.
Therefore, there exists $k_p\in \mathbb{N}$ such that $gh_p = h_p c^{k_p}$.
So we have that $g=h_p c^{k_p} h_p^{-1}$.

Moreover, the scaling factor $\mu_p$ of the action of $g$ on the tangent space at $p$ is 
\begin{align}
\label{eq:mu_p}
\mu_p (g) = \exp\left( 2\pi\sqrt{-1}\frac{k_p}{N} \right)
\end{align}
As before $N$ is the cardinality of the cyclic group $\ip{c}$.
The holomorphic Lefschetz fixed point formula gives
$$
\chi_{H^{0,1}}(g) = 
\begin{cases}
 \textrm{genus}(S) & \textrm{ if } g=id\\
 1-\sum_{g\cdot p=p} \frac 1 {1-\mu_p(g)}
\end{cases}
$$
For the group $\ip{c}$ we have the characters $\chi_j$ defined on the generator via
$$
\chi_j(c):= \exp\left( 2\pi\sqrt{-1}\frac j N \right)
$$
Then we have for any $j\in \mathbb{N}$ the following relation between local multipliers and characters:
$$
\sum_{g\cdot p=p} \mu_p(g)^j = \Ind^G_{\ip{c}} \chi_j (g)
$$
This follows from the description of $\mu_p$ in \autoref{eq:mu_p} and the formula for the induced character given in \autoref{eq:ind_rep}.

\paragraph{An auxiliary function.}
To compute the multiplicity of $\pi$ in $H^{0,1}$, we introduce an auxiliary function.
Its power series expansion will allow us to use the above formula for the powers of multipliers.
For $r\in \mathbb{C}$, define
$$
\chi_{H^{0,1}}(g,r) =
\begin{cases}
 \textrm{genus}(S) & \textrm{ if } g=id\\
 1-\sum_{g\cdot p=p} \frac 1 {1-r\mu_p(g)} & \textrm{ otherwise}
\end{cases}
$$
For fixed $g\in G$, this is a meromorphic function of $r$.
It has finitely many poles on the unit circle, but there is no pole at $1$.
The value of $\chi_{H^{0,1}}(g,r)$ at $r=1$ is exactly the character we want to understand.

We shall next perform a power series expansion for $|r|<1$ and manipulate the function as a uniformly convergent power series.
The result will have a limit as $r\to 1$ and this will give the desired multiplicity.

We have, for $g\neq id$ and $|r|<1$ that
\begin{align*}
 \chi_{H^{0,1}}(g,r) & = 1 - \sum_{g\cdot p =p} \sum_{i\geq 0} r^i \mu_p(g)^i\\
	& = 1- \sum_{i\geq 0} r^i \Ind^G_{\ip{c}} \chi_i(g)
\end{align*}
This means that in general, for all $g\in G$ and $|r|<1$ we have
\begin{align*}
 \chi_{H^{0,1}}(g,r)& = 1- \Bigg(\sum_{i\geq 0} r^i \Ind^G_{\ip{c}} \chi_i(g) \Bigg)+ 
 \delta_{id}\left(\frac{|G|}{N}\Big(\sum_{i\geq 0}r^i \Big)-1 + \textrm{genus}(S)  \right)\\
 & = 1- \left(\sum_{i\geq 0} r^i \Ind^G_{\ip{c}} \chi_i(g)\right) +\delta_{id} |G| \left( \frac 1 {1-r}\cdot \frac 1 N + \frac{N-1}{2N} \right)
\end{align*}

\paragraph{Multiplicities of representations.}
To compute the multiplicity of a representation $\pi$ in $H^{0,1}$, we take the product of the corresponding characters.
We shall use $\chi_{H^{0,1}}(r)$ with its power series expansion, and will evaluate the result as $r\to 1$.
We have the following expression (using Frobenius reciprocity at the last step):
\begin{multline*}
 \ip{\chi_{H^{0,1}}(r),\chi_\pi  }_G = \\
 = \ip{ 1- \left(\sum_{i\geq 0} r^i \Ind^G_{\ip{c}} \chi_i(g)\right) +\delta_{id} |G| \left( \frac 1 {1-r}\cdot \frac 1 N + \frac{N-1}{2N} \right) , \chi_\pi }_G\\
 = \delta_{\pi=triv} - \sum_{i\geq 0}r^i\ip{\Ind \chi_i,\chi_\pi}_G + \dim \pi \left(\frac 1 {1-r} \cdot \frac 1N + \frac{N-1}{2N} \right)\\
 = \delta_{\pi=triv} +\frac 12 \dim \pi - \sum_{i\geq 0} r^i \ip{\chi_i,\Res \chi_\pi}_{\ip{c}} + \frac {\dim \pi}{N} \left( \frac 1{1-r} - \frac 12 \right)
\end{multline*}
Define now $m_i:=\ip{\chi_i,\Res \chi_\pi}_{\ip{c}}$, which is also the number of eigenvalues of $c$ in the representation $\pi$ that are equal to $\exp(2\pi \sqrt{-1} \frac i N)$.
We then have that
$$
m_0+m_1 + \cdots + m_{N-1} = \dim \pi
$$
We also extend periodically the sequence by $m_{j+N}=m_j$.
The multiplicity of the representation then becomes
\begin{multline*}
 \delta_{\pi=triv} + \frac 12 \dim \pi - \sum_{i\geq 0} r^i m_i + \frac {\dim \pi}{N}\left( \frac 1 {r-1} - \frac 12 \right) =\\
=\delta_{\pi=triv} + \frac 12 \dim \pi - \frac{m_0 + r m_1 + \cdots + r^{N-1}m_{N-1}}{1-r^N} + \frac {\dim \pi}N \left(\frac 1 {1-r} - \frac 12\right)
\end{multline*}
We would like to evaluate the last two terms at $r=1$.
For this, we first rearrange:
\begin{multline*}
 \frac{m_0 + rm_1 + \cdots + r^{N-1}m_{N-1}} {1-r^N} - \frac{\dim \pi}{N}\cdot \frac 1 {1-r}=\\
 =\frac{1}{1-r}\left(\frac{ m_0 + rm_1+\cdots + r^{N-1}m_{N-1} }{1+r+\cdots+r^{N-1}} - \frac{\dim \pi}{N}\right)
\end{multline*}
Recalling that $m_0+\cdots+m_{N-1}=\dim \pi$, we can evaluate the limit as $r\to 1$ via L'Hopital's rule:
\begin{multline*}
 \left.\left( \frac{m_0+\cdots + r^{N-1}m_{N-1}}{1+r+\cdots+r^{N-1}} \right)^{'} \middle|_{r=1}\right. = \\
 = \frac {1}{N^2} \Bigg( \left(m_1+2m_2+\cdots + (N-1)m_{N-1} \right)N- \\
 -\frac{N(N-1)}2 (m_0+m_1+\cdots+m_{N-1})  \Bigg)\\
 =\frac 1N \left( \sum_{i=0}^{N-1} \left(i-\frac {N-1}2\right)m_i \right)
\end{multline*}
Plugging this value into the above calculation of the multiplicity, we find that $\pi$ appears in $H^{0,1}$ with multiplicity
\begin{align}
\label{eq:holom_mult}
 \delta_{\pi=triv} + \frac 12 (\dim \pi - m_0) + \frac 1N \left( \sum_{i=1}^{N-1} \left(i-\frac N2\right)m_i \right)
\end{align}
As usual, $\delta_{\pi=triv}$ is a constant equal to $1$ if $\pi$ is the trivial representation and $0$ otherwise.
The numbers $m_i$ denote the multiplicity of the eigenvalue $\exp \left(2\pi \sqrt{-1}\frac{i}{N}\right)$ for the element $c\in G$ in the $G$-representation $\pi$.

\begin{remark}
 \begin{enumerate}
  \item The term $\delta_{\pi=triv}+1/2(\dim \pi - m_0)$ from \autoref{eq:holom_mult} is exactly one half of what appears in \autoref{eq:top_mult} for the multiplicity of $\pi$ in $H^1$.
  \item Denote $\Delta_\pi:=\sum_{i=0}^{N-1} \left(i - N/2\right)m_i$.
  Then for the complex-conjugate representation $\overline{\pi}$ we have $\Delta_{\overline{\pi}} = - \Delta_\pi$.
  So the multiplicity in $H^{0,1}$ combined with that for $H^{1,0}$ gives the correct multiplicity for $H^1$.
  \item If $\chi_\pi$ is real-valued, then necessarily $\Delta_\pi = 0$.
  Indeed, in this case the representations $\pi$ and $\overline{\pi}$ are isomorphic since they have the same character.
  The character $\chi_\pi$ takes real values if and only if $\pi$ is a real or quarternionic representation.
  Therefore, $\Delta_\pi$ can be non-zero only for purely complex representations of the finite group $G$.
 \end{enumerate}

\end{remark}

\section{General considerations about monodromy}
\label{sec:monodromy}

This section contains a general discussion of two aspects relevant to our constructions in \autoref{s.quaternion-cover}. First, in \autoref{subsec:loc_syst} we discuss the general structure of semisimple local systems with a finite symmetry group.
The main consequences are for the structure of the Lyapunov spectrum and monodromy representation.
Next, in \autoref{subsec:SOstar} we discuss the group $SO^*(2n)$ in more detail.
In particular, we describe it both as a group of quaternionic and complex matrices.

\subsection{Local systems with symmetries}
\label{subsec:loc_syst}

\newcommand{\bV}{\mathbb{V}}
\newcommand{\bR}{\mathbb{R}}
\newcommand{\bC}{\mathbb{C}}
\newcommand{\bH}{\mathbb{H}}
\newcommand{\bW}{\mathbb{W}}

\paragraph{Setup.}
Consider a local system $\bW\to X$ over some base.
Assume that $\bW$ satisfies the usual semisimplicity properties: any sublocal system has a complement.
Consider the situation when a finite group $G$ acts on the fibers of the local system $\bW$.
We would like to understand the decomposition of $\bW$ for these symmetries, as well as the consequences for the Lyapunov spectrum (when the base $X$ carries a flow).

\paragraph{Isotypical components of representations.}
First, we consider the canonical decomposition of any $G$-representation $W$.
To fix notation, for an isomorphism class of irreducible representation $\pi$, let $R_\pi$ be a vector space with $G$-action realizing this isomorphism class.
Thus, we have a chosen map $G\to GL(R_\pi)$.

Further, let $A_\pi$ denote the algebra of endomorphisms of $R_\pi$ viewed as a $G$-representation.
When the field of scalars is $\mathbb{R}$, the possibilities for $A_\pi$ are $\mathbb{R},\mathbb{C}$ or $\bH$ (see \autoref{subsec:rep_th_prelims}).

For a given representation $W$, we can form \emph{the space of isotypical components} corresponding to $\pi$, defined by
$$
V_\pi := \textrm{Hom}_G(R_\pi,W)
$$
In other words, $V_\pi$ is the linear space of maps from $R_\pi$ to $W$ which commute with the $G$-action.
Note that $V_\pi$ \emph{does not} carry an action of $G$, however it does carry an action of $A_\pi$.
Indeed, given $a\in A_\pi, \phi\in V_\pi$ and $r\in R_\pi$, define
$$
a\cdot \phi (r) := \phi(ar)
$$
Note that this makes $V_\pi$ into a \emph{right} $A_\pi$-module, since $a_1 \cdot (a_2 \cdot \phi) = (a_2\cdot a_1)\cdot \phi$.
We shall therefore write the action of $A_\pi$ on $V_\pi$ on the right.

Next, we have a natural evaluation map:
\begin{align*}
 ev:& V_\pi\otimes R_\pi \to W\\
 & \phi\otimes r \mapsto \phi(r)
\end{align*}
Note that this map surjects onto the space of isotypical components of $W$ isomorphic to $R_\pi$.
However, it also has a kernel.
Namely, given $a\in A_\pi$ (recall that is acts on $V_\pi$ on the right), we have that
$$
ev(\phi\cdot a \otimes r) = \phi(ar) = ev(\phi\otimes a\cdot r)
$$
We therefore have a factorization of the evaluation map to
$$
ev: V_\pi \otimes_{A_\pi} R_\pi \to W
$$
where $V_\pi\otimes_{A_\pi} R_\pi := V_\pi\otimes R_\pi / \left\lbrace\phi a \otimes r = \phi\otimes ar\right\rbrace$.
Moreover, this map is a natural isomorphism onto the collection of isotypical components of $W$ isomorphic to $R_\pi$.

To summarize the discussion, recall we started with an arbitrary representation $W$ of $G$.
For each irreducible representation $\pi$ of $G$, we constructed the space $V_\pi$ which carried a right action of $A_\pi$ - the endomorphisms of $R_\pi$.
This gave us a natural isomorphism
$$
W=\bigoplus_{\pi\in\textrm{Irr}_{\mathbb{R}}(G)} V_\pi \otimes_{A_\pi} R_\pi
$$

\paragraph{Isotypical components of local systems.}
We now extend the above discussion to local systems.

Given the local system $\bW\to X$ carrying an action of the finite group $G$, we can form the local systems of isotypical components (for each representation $\pi$ of $G$)
$$
\bV_\pi := \textrm{Hom}_G(R_\pi,\bW)
$$
As before, $\bV_\pi$ carries an action of the endomorphisms $A_\pi$ of $R_\pi$, but no canonical action of the group $G$.
We have the canonical decomposition of $\bW$ as
\begin{align}
\label{eq:loc_syst_dec}
\bW = \bigoplus_{\pi\in\textrm{Irr}_{\mathbb{R}}(G)} \bV_\pi \otimes_{A_\pi} R_\pi
\end{align}
If $\bW$ carried a variation of Hodge structures, then so will $\bV_\pi$ (and the weight will be the same).
Moreover, now the variation on $\bV_\pi$ will carry an action of the endomorphism algebra $A_\pi$, rather than the group $G$ which was on $\bW$.

\paragraph{Lyapunov exponents. Monodromy.}
If the base of the local system $\bW\to X$ carried a flow, we can consider the associated Lyapunov exponents.
To understand them, recall the canonical decomposition provided in \autoref{eq:loc_syst_dec}.
Let $r$ be the rank of $R_\pi$ viewed as a module over $A_\pi$.
This will control the multiplicities of the answers for $\bW$ in terms of those for $\bV_\pi$.

Each local system $\bV_\pi$ will have its own Lyapunov exponents $\{\lambda_{i,\pi}\}$.
Then for $\bW$, the Lyapunov exponent $\lambda_{i,\pi}$ will appear with multiplicity $r$.
Moreover, the Oseledets subspaces for $\bW$ can be also read off from \autoref{eq:loc_syst_dec} and the Oseledets subspaces of $\bV_\pi$.
This is compatible with the fact that the Oseledets subspaces of $\bW$ must themselves carry an action of $G$.

Finally, the monodromy for $\bW$ can be understood in terms of the monodromy for $\bV_\pi$ as follows.
The local system $\bV_\pi$ will have some monodromy group $\mathbb{G}$, which will commute with the action of $A_\pi$.
Then the monodromy of $\bW$ will equal $r$ copies of the same group.

\paragraph{The example from \autoref{t.FFM}.}
In this case, we will have a local system, called $W_{\chi_2}$ in the sequel, but let us call it $\bW$ for this discussion.
It will correspond to a quaternionic representation $\chi_2$ occurring with multiplicity $3$.
This means the space of isotypical components of $\bW$ (denoted $\bV_{\chi_2}$ according to the above discussion) will be of rank $3$ over the endormorphisms of $\chi_2$, which are the quaternions $\bH$.

Since the real dimension of the representation will be $4$, the total real dimension of $\bW$ will be $12$.
We also see that the real dimension of $\bV_{\chi_2}$ will be $12$, but its monodromy can be viewed as quaternionic $3\times 3$ matrices.

\subsection{A description of \texorpdfstring{$SO^*(2n)$}{SOstar} }
\label{subsec:SOstar}
This section describes in more detail the group $SO^*(2n)$.
It is a real Lie group, but can be viewed as a group of either quaternionic or complex matrices, preserving certain (skew-)hermitian forms.

\paragraph{Conventions.}
Quaternions will be denoted
$$
\bH = \{ a + bi + cj + dk | a,b,c,d\in \bR\}
$$
If $x=a+bi+cj+dk$ is a quaternion, its conjugate is $\overline{x}:=a-bi-cj-dk$.
Elements of $\bH^n$ are quaternionic column vectors, and if $x\in \bH^n$ is one such, then $x^t$ denotes the corresponding row vector.
Denote adjoints of vectors by $x^\dag := \overline{x}^t$.

\paragraph{Quaternionic description of $SO^*(2n)$.}
View $\bH^n$ as a \emph{right} $\bH$-module and define the quaternionic-linear group by
$$
GL_n(\bH):= \{ f:\bH^n\to \bH^n \textrm{ invertible} | f(x\lambda) = f(x)\lambda, \forall x\in \bH^n, \lambda\in \bH\}
$$
We can define the quaternionic skew-hermitian form $C(-,-)$ for $x,y\in \bH^n$ via
$$
C(x,y) := x^\dag \cdot j \cdot y
$$
Here $j\in \bH$ is just the element of the quaternions, $x^\dag$ is a row vector and $y$ is a column vector.
Note that we have the following properties of the bilinear form, for $\lambda,\mu \in \bH$:
$$
C(x\cdot \lambda,y\cdot \mu) = -\overline{\lambda}\cdot \overline{C(y,x)} \cdot \mu
$$
The group is now defined by the requirement to preserve this form:
$$
SO^*(2n) := \{ f\in GL_n(\bH) | C(x,y) = C\left(f(x),f(y)\right) \}
$$
Another common name for this group is $U^*_n(\bH)$.

\paragraph{Complex description of $SO^*(2n)$.}
For a complex description, identify $x\in\bH$ with the vector $\left[\begin{smallmatrix} a_x\\ b_x \end{smallmatrix}\right]\in \bC^2$ by the requirement that $x=a_x + jb_x$ in $\bH$.
Further make the identification of $x\in\bH^n$ with $\left[\begin{smallmatrix} a_x\\ b_x \end{smallmatrix}\right]\in \bC^{2n}$ via $x=a_x + jb_x$ (thus $a_x,b_x\in \bC^n$).

Note that $\bC$ acts on both spaces on the right in a natural way, giving both the structure of a $\bC$ vector space.
On $\bC^{2n}$ we have a $\bC$-antilinear operator $R_j$ induced from multiplication by $j$ on the right:
$$
R_j \begin{bmatrix} a_x\\ b_x \end{bmatrix}= 
\begin{bmatrix} -\overline{b_x}\\ \overline{a_x} \end{bmatrix}
$$
For two quaternion vectors $x,y\in \bH^n $ and written as $x=a_x+jb_x$ and $y=a_y+jb_y$, we have
$$
C(x,y) = x^\dag \cdot j \cdot y = \left( b_x^\dag a_y - a_x^\dag b_y \right) + j \left( a_x^t a_y + b_x^t b_y \right)
$$
On $\bC^{2n}$ we can thus define two forms
\begin{align*}
\omega\left( \begin{bmatrix} a_x\\ b_x \end{bmatrix}, \begin{bmatrix} a_y\\ b_y \end{bmatrix}\right) &=
b_x^\dag a_y - a_x^\dag b_y \\
g\left( \begin{bmatrix} a_x\\ b_x \end{bmatrix}, \begin{bmatrix} a_y\\ b_y \end{bmatrix}\right) &=
a_x^t a_y + b_x^t b_y 
\end{align*}
Note that $\omega$ is skew-hermitian, i.e. $\omega(v,w)=-\overline{\omega(w,v)}$, and $g$ is symmetric.
Letting $h(v,w):=i\omega(v,w)$, we see that $h$ is in fact hermitian, i.e. $h(v,w)=\overline{h(w,v)}$ and a linear transformation of $\bC^{2n}$ preserves $h$ if and only if it preserves $\omega$.

It is clear that any quaternionic matrix preserving the quaternionic skew-hermitian form $C$ above, viewed as a matrix in $GL_{2n}(\bC)$, will necessarily preserve the symmetric form $g$ and hermitian form $h$.

However, the converse is also true.
If a matrix $A\in GL_{2n}(\bC)$ preserves both $g$ and $h$, then it must come from a quaternionic matrix preserving $C$.
For this, it suffices to check that such an $A$ will have to commute with the operator $R_j$ defined above.
More generally, we have that any two of the following conditions imply the third
\begin{itemize}
 \item The matrix $A\in GL_{2n}(\bC)$ preserves the symmetric form $g$.
 \item The matrix $A\in GL_{2n}(\bC)$ preserves the hermitian form $h$.
 \item The matrix $A\in GL_{2n}(\bC)$ commutes with the $\bC$-antilinear operator $R_j$.
\end{itemize}
In particular, we see that the group $SO^*(2n)$ can be alternatively described as the intersection of a complex orthogonal and a unitary group.
Writing out the matrices for $g$ and $h$ explicitly, we find that
$$
SO^*(2n) = \left \lbrace A\in GL_{2n}(\bC) \middle| AA^t = 1 \textrm{ and } A^\dag \cdot 
\begin{bmatrix}
 0 & i\\
 -i & 0
\end{bmatrix} A = 
\begin{bmatrix}
 0 & i\\
 -i & 0
\end{bmatrix}
\right \rbrace
$$

\section{An interesting quaternionic cover \texorpdfstring{$\widetilde{L}$}{widetildeL}}
\label{s.quaternion-cover}
In this section, we shall describe the origami $\widetilde{L}$.
Then, using known results about Lyapunov exponents, we shall compute its Lyapunov spectrum.
In this particular case, all exponents can be computed explicitly.
This will constrain the monodromy to some extent, and in \autoref{s.irreducibility-exotic-monodromy} we shall prove that in fact the non-trivial piece has $SO^*(2n)$ as its Zariski closure of the monodromy.

\subsection{Definition of \texorpdfstring{$\widetilde{L}$}{widetildeL} }
Let $L^0\in\mathcal{H}(2)$ be the $L$-shaped genus $2$ origami associated to the pair of permutations $h_0 = (1,2)(3)$ and $v_0 = (1,3)(2)$.

Consider the following covering $\widetilde{L}$ of $L^0$. For each element $g$ of the quaternion group $Q=\{1, -1, i, -i, j, -j, k, -k\}$, let us take a copy $L_g$ of $L^0$. The origami $\widetilde{L}$ is obtained by gluing (by translation) the two topmost horizontal sides of $L_g$ with the corresponding two bottom-most horizontal sides of $L_{gi}$, and the two rightmost vertical sides of $L_g$ with the corresponding two leftmost vertical sides of $L_{gj}$ (for each $g\in Q$). Alternatively, we label the sides of $L_g$ according to \autoref{f.1}, and we glue by translation the sides with the same labels.

\begin{figure}[htb!]
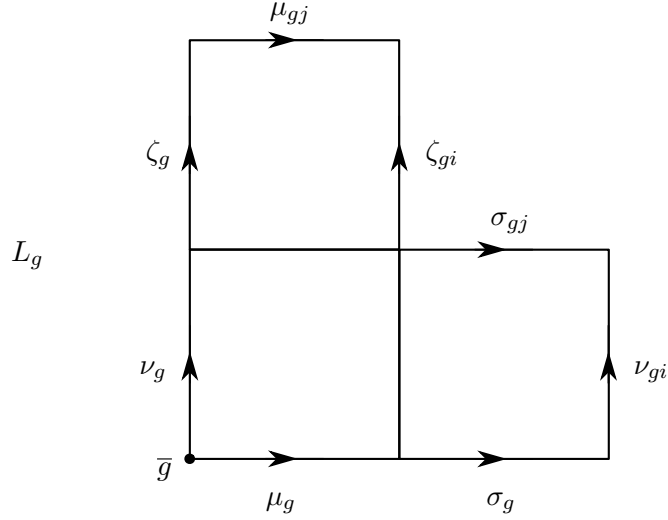
\caption{Labels of the sides of the $L_g$, $g\in Q$.}\label{f.1}
\end{figure}

Denoting by $\underline{g}\in \widetilde{L}$ the bottommost and leftmost corner of $L_g$ (as indicated in \autoref{f.1}), we have that $\underline{g}=\underline{-g}$ for each $g\in Q$, and, moreover, the set $\Sigma=\{\underline{1}, \underline{i}, \underline{j}, \underline{k}\}$ of four distinct points consists of all conical singularities of $\widetilde{L}$. Furthermore, it is not hard to check that the conical angle around each of these singularities is $12\pi$. Thus, $\widetilde{L}\in\mathcal{H}(5,5,5,5)$, so that
$\widetilde{L}$ has genus $11$.

\subsection{The group of automorphisms of \texorpdfstring{$\widetilde{L}$}{widetildeL}} Given $h\in Q$, we obtain an automorphism of $\widetilde{L}$ by sending (by translation) $L_g$ to $L_{hg}$. As it turns out, this accounts for all automorphisms of
$\widetilde{L}$, i.e., $\textrm{Aut}(\widetilde{L})\simeq Q$.

Recall that the non-trivial subgroups of $\textrm{Aut}(\widetilde{L})\simeq Q$ are its center $Z=\{1,-1\}$, and $\langle i\rangle = \{1,-1,i,-i\}$, $\langle j\rangle = \{1,-1,j,-j\}$ and
$\langle k\rangle = \{1,-1,k,-k\}$.

The quotient $L_{\pm}$ of $\widetilde{L}$ by $Z$ is an origami (of genus $5$) in $\mathcal{H}(2,2,2,2)$ with $\textrm{Aut}(L_{\pm})$ is isomorphic to Klein's group $\mathbb{Z}/2\mathbb{Z} \times \mathbb{Z}/2\mathbb{Z}\simeq Q/Z=\{1_{\pm}, i_{\pm},  j_{\pm}, k_{\pm}\}$.
Moreover, the quotients $L_{\langle i_{\pm}\rangle}$, $L_{\langle j_{\pm}\rangle}$ and $L_{\langle k_{\pm}\rangle}$ (resp.) of $L_{\pm}$ by the subgroups $\langle i_{\pm}\rangle=\{1_{\pm}, i_{\pm}\}$, $\langle j_{\pm}\rangle=\{1_{\pm}, j_{\pm}\}$ and $\langle k_{\pm}\rangle=\{1_{\pm}, k_{\pm}\}$ (resp.) of $\textrm{Aut}(L_{\pm})\simeq Q/Z$ give rise to three origamis (of genus $3$) in $\mathcal{H}(2,2)$ that are precisely the quotients of $\widetilde{L}$ by the subgroups $\langle i\rangle$, $\langle j\rangle$ and $\langle k\rangle$ of $\textrm{Aut}(\widetilde{L})\simeq Q$.
Furthermore, the origamis $L_{\langle i_{\pm}\rangle}$, $L_{\langle j_{\pm}\rangle}$ and $L_{\langle k_{\pm}\rangle}$ are unramified double covers of the origami $L^0\in\mathcal{H}(2)$.
In some sense, these characteristics of the origami $L_{\pm}$ are similar to the so-called \emph{wind-tree model} (cf. \cite{DHL}) except that \emph{all} origamis $L_{\langle i_{\pm}\rangle}$, $L_{\langle j_{\pm}\rangle}$ and $L_{\langle k_{\pm}\rangle}$ belong to the odd connected component $\mathcal{H}(2,2)^{odd}$ of the stratum $\mathcal{H}(2,2)$.

\subsection{The first absolute homology group of \texorpdfstring{$\widetilde{L}$}{widetildeL}}
The intermediate cover $\widetilde{L}\to L_{\pm} = \widetilde{L}/Z$ of $\widetilde{L}\to\mathbb{T}^2$ induces a decomposition $$H_1(\widetilde{L}, \mathbb{R}) = H_1^+(\widetilde{L}, \mathbb{R}) \oplus H_1^-(\widetilde{L}, \mathbb{R})$$ where $-1\in Z\subset Q\simeq \textrm{Aut}(\widetilde{L})$ acts by $\textrm{id}$, resp. $-\textrm{id}$, on $H_1^+(\widetilde{L}, \mathbb{R})$, resp. $H_1^-(\widetilde{L}, \mathbb{R})$, and $H_1^+(\widetilde{L})$ is naturally isomorphic to $H_1(L_{\pm},\mathbb{R})$. Note that $H_1^+(\widetilde{L},\mathbb{R})$ is a $10$-dimensional subspace of the $22$-dimensional space $H_1(\widetilde{L}, \mathbb{R})$ (since $L_{\pm}$ has genus $5$ and $\widetilde{L}$ has genus $11$).

Similarly, the intermediate covers $L_{\pm}\to L_{\langle i_{\pm} \rangle}\to L^0$, $L_{\pm}\to L_{\langle j_{\pm}\rangle}\to L^0$ and $L_{\pm}\to L_{\langle k_{\pm}\rangle}\to L^0$ induce a decomposition
$$H_1(L_{\pm},\mathbb{R}) = H_1^{+,+}(L_{\pm}, \mathbb{R}) \oplus H_1^{+,-}(L_{\pm}, \mathbb{R}) \oplus H_1^{-,+}(L_{\pm}, \mathbb{R}) \oplus H_1^{-,-}(L_{\pm}, \mathbb{R})$$
where $i_{\pm}\in\textrm{Aut}(L_{\pm})$, resp. $j_{\pm}\in \textrm{Aut}(L_{\pm})$, acts on $H_1^{\varepsilon_{i_{\pm}}, \varepsilon_{j_{\pm}}}(L_{\pm},\mathbb{R})$, $\varepsilon_{i_{\pm}}, \varepsilon_{j_{\pm}}\in \{+,-\}$ by $(\varepsilon_{i_{\pm}})\textrm{id}$, resp. $(\varepsilon_{j_{\pm}})\textrm{id}$, and
\begin{itemize}
\item $H_1^{+,+}(L_{\pm}, \mathbb{R}) \simeq H_1(L^0,\mathbb{R}) = H_1^{st}(L^0, \mathbb{R})\oplus H_1^{(0)}(L^0, \mathbb{R})$,
\item $H_1^{+,+}(L_{\pm}, \mathbb{R})\oplus H_1^{+,-}(L_{\pm}, \mathbb{R})\simeq H_1(L_{\langle i_{\pm} \rangle}, \mathbb{R})$,
\item $H_1^{+,+}(L_{\pm}, \mathbb{R})\oplus H_1^{-,+}(L_{\pm}, \mathbb{R})\simeq H_1(L_{\langle j_{\pm} \rangle}, \mathbb{R})$,
\item $H_1^{+,+}(L_{\pm}, \mathbb{R})\oplus H_1^{-,-}(L_{\pm}, \mathbb{R})\simeq H_1(L_{\langle k_{\pm} \rangle}, \mathbb{R})$.
\end{itemize}

By means of the isomorphism $H_1(L_{\pm}, \mathbb{R})\simeq H_1^+(\widetilde{L}, \mathbb{R})$, we obtain a decomposition
$$H_1(\widetilde{L}, \mathbb{R})\simeq H_1^{st}(L^0,\mathbb{R})\oplus H_1^{(0)}(L^0,\mathbb{R}) \oplus\bigoplus\limits_{\substack{\alpha,\beta\in\{+,-\}, \\ (\alpha,\beta)\neq (+,+)}} H_1^{\alpha,\beta}(L_{\pm}, \mathbb{R})\oplus H_1^-(\widetilde{L}, \mathbb{R})$$
where all summands are symplectic subspaces that are mutually symplectically orthogonal, and all summands are $2$-dimensional subspaces of $H_1(\widetilde{L}, \mathbb{R})$ except for the $12$-dimensional subspace $H_1^-(\widetilde{L}, \mathbb{R})$.

Observe that the action of $\textrm{Aff}(\widetilde{L})$ on $H_1(\widetilde{L}, \mathbb{R})$ respects each summand of the decomposition above: these summands were defined in terms of deck transformations of certain intermediate coverings of $\widetilde{L}\to\mathbb{T}^2$ which act by pre-composition with translation charts of $\widetilde{L}$, and $\textrm{Aff}(\widetilde{L})$ acts by post-composition with translation charts of $\widetilde{L}$.
In particular, this is a decomposition of the $\textrm{Aut}(\widetilde{L})$-module $H_1(\widetilde{L},\mathbb{R})$ into $\textrm{Aut}(\widetilde{L})$-submodules.

Note that this decomposition of $H_1(\widetilde{L}, \mathbb{R})$ refines $H_1(\widetilde{L}, \mathbb{R}) = H_1^{st}(\widetilde{L}, \mathbb{R}) \oplus H_1^{(0)}(\widetilde{L}, \mathbb{R})$ in the sense that
\begin{equation}\label{e.homology-decomposition}H_1^{(0)}(\widetilde{L},\mathbb{R}) \simeq H_1^{(0)}(L^0,\mathbb{R})\oplus\bigoplus\limits_{\substack{\alpha,\beta\in\{+,-\}, \\ (\alpha,\beta)\neq (+,+)}} H_1^{\alpha,\beta}(L_{\pm}, \mathbb{R})\oplus H_1^-(\widetilde{L}, \mathbb{R})
\end{equation}

As it turns out, this is precisely the decomposition of $H_1^{(0)}(\widetilde{L},\mathbb{R})$ into isotypical components. More precisely, the quaternion group $\textrm{Aut}(\widetilde{L})\simeq Q$ has five irreducible representations $\chi_1, \chi_i, \chi_j, \chi_k, \chi_2'$ (over $\mathbb{C}$) whose characters are given by the following table:
\begin{center}
{\renewcommand{\arraystretch}{1.5}
\renewcommand{\tabcolsep}{0.2cm}
\begin{tabular}{|c|c|c|c|c|c|}
\hline
 & $1$ & $-1$ & $\pm i$ & $\pm j$ & $\pm k$ \\
\hline
$\chi_1$ & $1$ & $1$ & $1$ & $1$ & $1$ \\
\hline
 $\chi_i$ & $1$ & $1$ & $1$ & $-1$ & $-1$\\
\hline
$\chi_j$ & $1$ & $1$ & $-1$ & $1$ & $-1$\\
\hline
$\chi_k$ & $1$ & $1$ & $-1$ & $-1$ & $1$\\
\hline
tr $\chi_2'$ & $2$ & $-2$ & $0$ & $0$ & $0$\\
\hline
\end{tabular}}
\end{center}
Furthermore, the representation $\chi_2'$ is quaternionic while the representations $\chi_1$, $\chi_i$, $\chi_j$, $\chi_k$ are real.
A quick comparison between this table and the actions of the elements of $Q$ on the summands of this decomposition of $H_1^{(0)}(\widetilde{L}, \mathbb{R})$ reveals that
$$H_1^{(0)}(L^0,\mathbb{R}) = W_{\chi_1}\simeq 2\chi_1, \quad H_1^{+,-}(L_{\pm}, \mathbb{R}) = W_{\chi_i} \simeq 2\chi_i,$$
$$H_1^{-,+}(L_{\pm}, \mathbb{R}) = W_{\chi_j} \simeq 2\chi_j, \quad H_1^{-,-}(L_{\pm}, \mathbb{R}) = W_{\chi_k} \simeq 2\chi_k,$$
$$H_1^-(\widetilde{L}, \mathbb{R}) = W_{\chi_2}\simeq 3\chi_2,$$
where $W_a$ stands for the isotypical component of $a$ in $H_1^{(0)}(\widetilde{L},\mathbb{R})$,  $W_a\simeq \ell_a a$ with $\ell_a\in\mathbb{N}$ means that $a$ appears with multiplicity $\ell_a$ in $H_1^{(0)}(\widetilde{L}, \mathbb{R})$, and $\chi_2:=2\chi_2'$.

\subsection{Lyapunov spectrum of the KZ cocycle over \texorpdfstring{$SL(2,\mathbb{R})\cdot\widetilde{L}$}{SL2cdot}}
The Lyapunov spectrum of the restriction to $H_1^+(\widetilde{L},\mathbb{R})$ of the KZ cocycle over $SL(2,\mathbb{R})\cdot\widetilde{L}$ is not difficult to understand thanks to the intermediate covers $L_{\pm}\to L_{\langle \ast\rangle}\to L^0$, $\ast\in\{i_{\pm}, j_{\pm}, k_{\pm}\}$, and some results of Bainbridge and Chen-M\"oller.

More concretely, since $L^0\in\mathcal{H}(2)$, the work of Bainbridge \cite{B} ensures that the subbundles $H_1^{st}(\widetilde{L},\mathbb{R})$, resp. $W_{\chi_1}$ of 
$$H_1^{st}(\widetilde{L},\mathbb{R})\oplus W_{\chi_1} \simeq H_1^{+,+}(L_{\pm},\mathbb{R})\simeq H_1(L^0,\mathbb{R}) = H_1^{st}(L^0,\mathbb{R})\oplus H_1^{(0)}(L^0,\mathbb{R})$$ contribute with the Lyapunov exponents $1$ and $-1$, resp. $1/3$ and $-1/3$.

Next, since $L_{\langle\ast\rangle}\in \mathcal{H}(2,2)^{odd}$, $\ast\in\{i_{\pm}, j_{\pm}, k_{\pm}\}$, the work of Chen-M\"oller \cite{CM} guarantees that the sum of the non-negative Lyapunov exponents associated to each of the subbundles 
$$H_1(L_{\langle\ast\rangle},\mathbb{R})\simeq H_1^{+,+}(L_{\pm}, \mathbb{R}) \oplus H_1^{\alpha,\beta}(L_{\pm},\mathbb{R}) \simeq (H_1^{st}(\widetilde{L}, \mathbb{R}) \oplus W_{\chi_1})\oplus W_{\rho}$$
for the appropriate choices of $\alpha,\beta\in\{+,-\}$, $(\alpha,\beta)\neq (+,+)$, and $\rho\in\{\chi_i, \chi_j,\allowbreak \chi_k\}$ depending on $\ast\in\{i_{\pm}, j_{\pm}, k_{\pm}\}$) is $5/3$.
By combining this information with our knowledge of the Lyapunov exponents associated to $H_1^{+,+}(L_{\pm},\mathbb{R})$,
we deduce that each of the subbundles $W_{\rho}$, $\rho\in\{\chi_i, \chi_j, \chi_k\}$, contribute with Lyapunov exponents $1/3$ and $-1/3$. 

In summary, the Lyapunov spectrum of the restriction of the KZ cocycle over $SL(2,\mathbb{R})\cdot \widetilde{L}$ to the $10$-dimensional symplectic subspace $H_1^+(\widetilde{L}, \mathbb{R})$ is 
$$1 > \frac{1}{3} = \frac{1}{3} = \frac{1}{3} = \frac{1}{3} > -\frac{1}{3} = -\frac{1}{3} = -\frac{1}{3} = -\frac{1}{3} > -1$$
\begin{remark}
It is possible to show that $\textrm{Aff}(\widetilde{L})$ acts on each $W_{\rho}$, $\rho\in\{\chi_1, \chi_i,\allowbreak \chi_j, \chi_k\}$, through a Zariski dense subgroup of $SL(2,\mathbb{R})$.
In particular, the monodromy of the restriction of the KZ cocycle over $SL(2,\mathbb{R})\cdot \widetilde{L}$ to each $W_{\rho}$, $\rho\in\{\chi_1, \chi_i, \chi_j, \chi_k\}$, is given by item (i) (with $d=1$) in the first author's classification of KZ cocycle monodromies mentioned in \autoref{s.introduction}. 
\end{remark}

The Lyapunov spectrum of the restriction of the KZ cocycle over $SL(2,\mathbb{R})\cdot \widetilde{L}$ to $H_1^-(\widetilde{L},\mathbb{R}) = W_{\chi_2}\simeq 3\chi_2$ is determined by the quaternionic nature of the symplectic $\textrm{Aut}(\widetilde{L})$-module $W_{\chi_2}\simeq 3\chi_2$ and the sum of the non-negative Lyapunov exponents of KZ cocycle over $SL(2,\mathbb{R})\cdot\widetilde{L}$. 

More concretely, since $\chi_2$ is a quaternionic representation, each Lyapunov exponent associated to $W_{\chi_2}$ has multiplicity four (at least).
Because $\textrm{Aff}(\widetilde{L})$ respects the symplectic intersection form on the symplectic module $W_{\chi_2}\simeq 3\chi_2$, it follows that the Lyapunov spectrum of the restriction of the KZ cocycle over $SL(2,\mathbb{R})\cdot \widetilde{L}$ to $W_{\chi_2}$ has the form 
$$\lambda = \lambda = \lambda = \lambda \geq 0 = 0 = 0 = 0 \geq -\lambda = -\lambda = -\lambda = -\lambda$$

Hence, the sum $\theta_1+\dots+\theta_{11}$ of the $11$ non-negative Lyapunov exponents of the KZ cocycle acting on $H_1(\widetilde{L}, \mathbb{R}) = H_1^+(\widetilde{L}, \mathbb{R}) \oplus H_1^-(\widetilde{L}, \mathbb{R})$ is 
$$\theta_1+\dots+\theta_{11} = 1 + 4\times\frac{1}{3} + 4\times\lambda + 2\times 0 = \frac{7}{3} + 4\lambda$$

In the sequel, we will use the Eskin-Kontsevich-Zorich formula \cite{EKZ1} in order to determine the numerical value of $\lambda$: 

\begin{theorem}\label{t.EKZ-L-tilde} The numerical value of the sum of the non-negative Lyapunov exponents of the KZ cocycle over $SL(2,\mathbb{R})\cdot\widetilde{L}$ is $3$. In particular, $\lambda=1/6$.
\end{theorem}

\begin{proof} Recall that Eskin-Kontsevich-Zorich \cite{EKZ1} showed that the sum of the non-negative Lyapunov exponents of KZ cocycle over the $SL(2,\mathbb{R})$-orbit of an origami $X\in\mathcal{H}(k_1,\dots, k_{\sigma})$ is given by 
$$\frac{1}{12}\sum\limits_{l=1}^{\sigma} \frac{k_l(k_l+2)}{k_l+1} + \frac{1}{\# SL(2,\mathbb{Z})\cdot X} \sum\limits_{\substack{Y\in SL(2,\mathbb{Z})\cdot X,
\\ c \textrm{ is a cycle of } h_Y}} \frac{1}{\textrm{length of } c}$$
where $(h_Y, v_Y)$ is a pair of permutations associated to the origami $Y$.

Since $\widetilde{L}\in \mathcal{H}(5,5,5,5)$, we obtain that $\theta_1+\dots+\theta_{11}$ is equal to 
\begin{eqnarray}\label{e.EKZ-L-tilde}
& & \frac{1}{12}\left(4\times \frac{5\times 7}{6}\right) + \frac{1}{\# SL(2,\mathbb{Z})\cdot \widetilde{L}} \sum\limits_{\substack{Y\in SL(2,\mathbb{Z})\cdot \widetilde{L},
\\ c \textrm{ is a cycle of } h_Y}} \frac{1}{\textrm{length of } c} \\ 
&=& \frac{35}{18} + \frac{1}{\# SL(2,\mathbb{Z})\cdot \widetilde{L}} \sum\limits_{\substack{Y\in SL(2,\mathbb{Z})\cdot \widetilde{L},
\\ c \textrm{ is a cycle of } h_Y}} \frac{1}{\textrm{length of } c} \nonumber
\end{eqnarray}
Therefore, our task is reduced to the computation of the $SL(2,\mathbb{Z})$-orbit of $\widetilde{L}$. 

For this sake, we note that the squares of $\widetilde{L}$ can be labeled in such a way that the square of $L_1$ in \autoref{f.1} with sides $\mu_1$ and $\nu_1$ has number $1$ and the remaining squares of $\widetilde{L}$ are numbered in a compatible manner with the pair of permutations 
\begin{align*}
\begin{split}
h_{\widetilde{L}}=(1,2,13,14,7,8,19,20)(3,15,9,21)(4,5,22,23,10,11,16,17)\\
(6,24,12,18)
\end{split}\\
\intertext{and}
\begin{split}
 v_{\widetilde{L}} = (1,3,4,6,7,9,10,12)(2,5,8,11)(13,15,16,18,19,21,22,24)\\
(14,17,20,23)
\end{split}
\end{align*}

In this setting, recall that the $SL(2,\mathbb{Z})$-orbit of $\widetilde{L}$ is obtained by successively applying the generators $T = \left(\begin{array}{cc} 1 & 1 \\ 0 & 1 \end{array}\right)$ and $S = \left(\begin{array}{cc} 1 & 0 \\ 1 & 1 \end{array}\right)$ of $SL(2,\mathbb{Z})$ while keeping in mind that $T$ and $S$ act on pair of permutations as $T(r,u)=(r,ur^{-1})$ and $S(r,u)=(ru^{-1}, u)$ and the pairs of permutations $(r,u)$ and $(\phi r\phi^{-1}, \phi u\phi^{-1})$ define the same origami. 

By performing this straightforward computation, one gets that $SL(2,\mathbb{Z})\cdot\widetilde{L}$ has the structure described by \autoref{f.2}, namely:
\begin{itemize}
\item $SL(2,\mathbb{Z})\cdot\widetilde{L} = \{\widetilde{L}, M_1, \dots, M_{11}\}$ has cardinality $12$;
\item $M_1=T(\widetilde{L})$, $M_2=T^2(\widetilde{L})$, $M_3=T^3(\widetilde{L})$, $M_4=S(\widetilde{L})$, $M_5=T(M_4)$, $M_6=S(M_4)$, $M_7=T(M_6)$, $M_8=T^2(M_6)$, $M_9=T^3(M_6)$, $M_{10}=S(M_2)$ and $M_{11}=S(M_6)$; 
\item $T$-orbits are $\{\widetilde{L}, M_1, M_2, M_3\}$, $\{M_4, M_5\}$, $\{M_6, M_7, M_8, M_9\}$, $\{M_{10}\}$ and $\{M_{11}\}$; 
\item $S$-orbits are $\{\widetilde{L}, M_4, M_6, M_{11}\}$, $\{M_1, M_9\}$, $\{M_2, M_{10}, M_8, M_5\}$, $\{M_3\}$ and $\{M_7\}$; 
\item the origamis $M_l$, $l\in\{4, 6, 10, 11\}$ are associated to pairs of permutations $(h_{M_l}, v_{M_l})$, $l\in\{4, 6, 10, 11\}$ with 
\begin{align*}
\begin{split}
 h_{M_4} = (1, 18, 17, 7, 24, 23) (2, 16, 9, 8, 22, 3) (4, 15, 14, 10, 21, 20) \\
	 (5, 13, 12, 11, 19, 6)\\
\end{split}\\
\begin{split}
h_{M_6} = 	(1, 11, 22, 20, 7, 5, 16, 14) (2, 19, 17, 10, 8, 13, 23, 4) (3, 18, 9, 24)\\
	  (6, 15, 12, 21)\\
\end{split}\\
\begin{split}
 h_{M_{10}}=	(1, 24, 11) (2, 4, 21) (3, 14, 16) (5, 7, 18) (6, 23, 13) (8, 10, 15) \\
	    	(12, 17, 19) (20, 22, 9)
\end{split}\\
\intertext{and}
\begin{split}
 h_{M_{11}} = (1, 21, 17) (2, 22, 6) (3, 11, 13) (4, 18, 14) (5, 19, 9) (7, 15, 23) \\
 (8, 16, 12) (10, 24, 20)
\end{split}
\end{align*}
\end{itemize}

\begin{figure}[htb!]
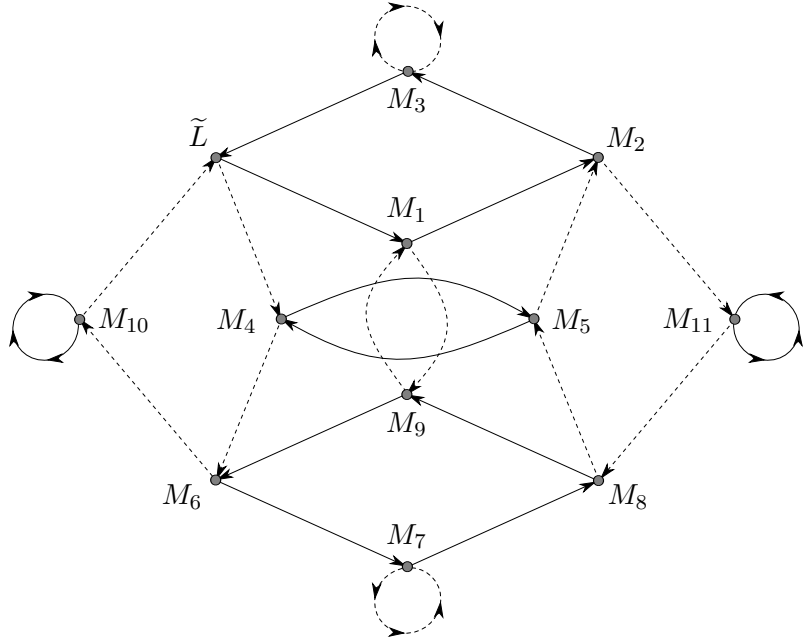
\caption{$SL(2,\mathbb{Z})$-orbit of $\widetilde{L}$. The full arrows indicate the action of $T$ and the dashed arrows indicate the action of $S$.}\label{f.2}
\end{figure}
Observe that we have that the first permutation $h_Y$ in a pair $(h_Y, v_Y)$ associated to an origami $Y$ is constant along the $T$-orbit of $Y$, i.e., $h_Z = h_Y$ whenever the origamis $Z$ and $Y$ belong to the same $T$-orbit (since $T(r,u) = (r, u r^{-1})$). It follows that 
\begin{eqnarray}\label{e.SL2Z-L-tilde}
& &\frac{1}{\# SL(2,\mathbb{Z})\cdot \widetilde{L}} \sum\limits_{\substack{Y\in SL(2,\mathbb{Z})\cdot \widetilde{L},
\\ c \textrm{ is a cycle of } h_Y}} \frac{1}{\textrm{length of } c}  = \\ 
& &\frac{1}{12}\left\{4\times\left(\frac{1}{8} + \frac{1}{4} + \frac{1}{8} + \frac{1}{4} \right) + 2\times\left(\frac{1}{6} + \frac{1}{6} + \frac{1}{6} + \frac{1}{6}\right) \right. \nonumber \\ & & \left. + 4\times\left( \frac{1}{8} + \frac{1}{8} + \frac{1}{4} + \frac{1}{4}\right) + \left( \frac{1}{3} + \frac{1}{3} + \frac{1}{3} + \frac{1}{3} + \frac{1}{3} + \frac{1}{3} + \frac{1}{3} + \frac{1}{3}\right) \right. \nonumber \\ & & \left. + \left( \frac{1}{3} + \frac{1}{3} + \frac{1}{3} + \frac{1}{3} + \frac{1}{3} + \frac{1}{3} + \frac{1}{3} + \frac{1}{3}\right) \right\} \nonumber \\ & & = \frac{19}{18} \nonumber
\end{eqnarray}

By combining \eqref{e.EKZ-L-tilde} and \eqref{e.SL2Z-L-tilde}, we conclude that 
$$\theta_1+\dots+\theta_{11} = \frac{35}{18} + \frac{19}{18} = 3$$
This completes the proof of the theorem.
\end{proof}

\begin{remark} The homological dimension of the unique $SL(2,\mathbb{R})$-probability measure supported on $SL(2,\mathbb{R})\cdot \widetilde{L}$ in the sense of Forni's paper \cite{Fo11} is $8$: this is not hard to deduce from the description of $SL(2,\mathbb{Z})\cdot\widetilde{L}$ given above. In particular, Forni's criterion in \cite{Fo11} says that the KZ cocycle over $SL(2,\mathbb{R})\cdot\widetilde{L}$ has $8$ positive Lyapunov exponents at least. In other terms, Forni's criterion falls short from predicting the correct number (i.e., $9$) of positive Lyapunov exponents in the particular case of the KZ cocycle over $SL(2,\mathbb{R})\cdot\widetilde{L}$. 
\end{remark}

\subsection{A dicothomy for the monodromy of the KZ cocycle on \texorpdfstring{$W_{\chi_2}$}{Wchi2}}\label{ss.dichotomy}
Recall from Remark \ref{r.semisimplicity} that we can split $W_{\chi_2}$ into a direct sum of $\textrm{Aff}_{***}(\widetilde{L})$-irreducible symplectic $\textrm{Aut}(\widetilde{L})$-submodules, where $\textrm{Aff}_{***}(\widetilde{L})$ is a finite-index subgroup of $\textrm{Aff}(\widetilde{L})$. Since $W_{\chi_2}\simeq 3\chi_2$, this gives us three possibilities: 
\begin{itemize}
\item $\textrm{Aff}_{***}(\widetilde{L})$ acts irreducibily on $W_{\chi_2}$;
\item $W_{\chi_2} = U \oplus V$ where $U$ and $V$ are $\textrm{Aff}_{***}(\widetilde{L})$-irreducible $\textrm{Aut}(\widetilde{L})$-submodules isomorphic to $U\simeq 2\chi_2$ and $V\simeq \chi_2$;
\item $W_{\chi_2} = W'\oplus W''\oplus W'''$ where $W'$, $W''$ and $W'''$ are $\textrm{Aff}_{***}(\widetilde{L})$-irreducible $\textrm{Aut}(\widetilde{L})$-submodules isomorphic to $W'\simeq W''\simeq W'''\simeq \chi_2$.
\end{itemize}
Furthermore, the discussion in \autoref{subsec:loc_syst}, as well as the classification of the monodromy groups of the KZ cocycle in \cite{Filip2} imply that the Zariski closure of the monodromy group (modulo compact and finite-index factors) of the restriction of the KZ cocycle to $W_{\chi_2}$ is:
\begin{itemize}
\item $SO^*(6)$ if $\textrm{Aff}_{***}(\widetilde{L})$ acts irreducibly on $W_{\chi_2}$; 
\item a subgroup of $SO^*(4)\times SO^*(2)$ if $W_{\chi_2} = U\oplus V$ for some $\textrm{Aff}_{***}(\widetilde{L})$-irreducible $\textrm{Aut}(\widetilde{L})$-submodules $U$ and $V$;
\item a subgroup of $SO^*(2)\times SO^*(2)\times SO^*(2)$ if $W_{\chi_2} = W'\oplus W''\oplus W'''$ for some $\textrm{Aff}_{***}(\widetilde{L})$-irreducible $\textrm{Aut}(\widetilde{L})$-submodules $W'$, $W''$ and $W'''$.
\end{itemize} 

We affirm that the possibility in the last item above can not occur.
Indeed, the situation described by this item would imply that all Lyapunov exponents in $W_{\chi_2}$ vanish, a contradiction with \autoref{t.EKZ-L-tilde} saying that the Lyapunov spectrum of the KZ cocycle on $W_{\chi_2}$ contains $\lambda = 1/6$. 

In summary, we showed the following result:
\begin{theorem}\label{t.FFM1} The Zariski closure of the monodromy group of the restriction of the KZ cocycle to $W_{\chi_2}$ is given by one of the next two possibilities:
\begin{itemize}
\item either $SO^*(6)$ if $\textrm{Aff}_{***}(\widetilde{L})$ acts irreducibly on $W_{\chi_2}$,  
\item or a subgroup of $SO^*(4)\times SO^*(2)$ if $W_{\chi_2} = U\oplus V$ for some $\textrm{Aff}_{***}(\widetilde{L})$-irreducible $\textrm{Aut}(\widetilde{L})$-submodules $U\simeq 2\chi_2$ and $V\simeq \chi_2$.
\end{itemize}
\end{theorem}

\begin{remark}
This result suffices to show that an ``exotic'' monodromy group of the type described in item (iv) of \autoref{s.introduction} occurs for the restriction to $W_{\chi_2}$ of the KZ cocycle over $SL(2,\mathbb{R})\cdot\widetilde{L}$. Nevertheless, we will complete the statement of \autoref{t.FFM1} by showing (in next section) that $\textrm{Aff}_{***}(\widetilde{L})$ acts irreducibly on $W_{\chi_2}$.
\end{remark}

\section{\texorpdfstring{$\textrm{Aff}_{***}(\widetilde{L})$}{Aff***(widetildeL)}-irreducibility of \texorpdfstring{$W_{\chi_2}$}{chi2} }
\label{s.irreducibility-exotic-monodromy}

In this section, we improve \autoref{t.FFM1} by showing the following result:

\begin{theorem}\label{t.FFM2} Any finite-index subgroup $\textrm{Aff}_{***}(\widetilde{L})$ of $\textrm{Aff}(\widetilde{L})$ acts irreducibly on $W_{\chi_2}$. In particular, the monodromy group of the KZ cocycle on $W_{\chi_2}$ is $SO^*(6)$.
\end{theorem}

Note that \autoref{t.FFM} is a direct consequence of this theorem, so that our task is reduced to prove \autoref{t.FFM2}.
In this direction, we will compute the action on $W_{\chi_2}$ of certain Dehn multitwists in an adequate basis of $W_{\chi_2}$. 

\subsection{A choice of basis of \texorpdfstring{$W_{\chi_2}$}{Wchi2}} The cycles $\mu_g$, $\sigma_g$, $\nu_g$, $\zeta_g$ displayed in \autoref{f.1} form a generating set of the relative homology group $H_1(\widetilde{L}, \Sigma,\mathbb{R})$. Note that we have a relation 
$$\square_g:=\mu_g + \sigma_g + \nu_{gi} - \sigma_{gj} + \zeta_{gi} - \mu_{gj} - \zeta_g - \nu_g = 0$$ 
for each $L_g$, $g\in Q$. Observe that this gives us $7$ independent\footnote{In particular, we can extract from the generating set $\{\mu_g, \sigma_g, \nu_g, \zeta_g\}_{g\in Q}$ of $32$ cycles a basis of $H_1(\widetilde{L}, \Sigma, \mathbb{R})$ with $32-7=25$ cycles. Of course, this is coherent with the facts that $\widetilde{L}$ has genus $11$ and $\#\Sigma=4$, so that $H_1(\widetilde{L}, \Sigma, \mathbb{R})$ has dimension $2\times 11 + 4 - 1 = 25$.} relations because $\sum\limits_{g\in Q}\square_g=0$. 

It is not difficult to check that, for each $g\in Q$, the cycles $\sigma_g$ and $\zeta_g$ are absolute cycles, and the cycles $\mu_g$ and $\nu_g$ are relative cycles with boundaries
$$\partial\mu_g = \overline{gi} - \overline{g} \quad \textrm{ and } \quad \partial\nu_g = \overline{gj} - \overline{g}$$

For each $g\in Q$, let us consider the following absolute cycles:
$$\widehat{\mu}_g := \mu_g - \mu_{-g}, \quad \widehat{\sigma}_g := \sigma_g - \sigma_{-g},$$ $$\widehat{\nu}_g := \nu_g - \nu_{-g} \quad \textrm{ and } \quad \widehat{\zeta}_g := \zeta_g - \zeta_{-g}$$
By definition, $\{\widehat{\mu}_g, \widehat{\sigma}_g, \widehat{\nu}_g, \widehat{\zeta}_g\}_{g\in Q}$ is a generating set for $W_{\chi_2}$, and it is not hard to verify that the list 
$$\mathcal{B} = \{\widehat{\sigma}_1, \widehat{\sigma}_i, \widehat{\sigma}_j, \widehat{\sigma}_k, \widehat{\zeta}_1, \widehat{\zeta}_i, \widehat{\zeta}_j, \widehat{\zeta}_k, \widehat{\mu}_1, \widehat{\mu}_i, \widehat{\nu}_1, \widehat{\nu}_j\}$$ 
of $12$ absolute cycles is a basis of $W_{\chi_2}\simeq 3\chi_2$. 

For later use, let us express $\widehat{\mu}_j$, $\widehat{\mu}_k$, $\widehat{\nu}_i$, $\widehat{\nu}_k$ in terms of the basis $\mathcal{B}$. In this direction, we expand the relations $\square_g-\square_{-g}=0$ for $g=1, i, j, k$: 
\begin{equation}\label{e.relation-hat-1}
\widehat{\mu}_1 + \widehat{\sigma}_1 +\widehat{\nu}_i - \widehat{\sigma}_j + \widehat{\zeta}_i - \widehat{\mu}_j - \widehat{\zeta}_1 - \widehat{\nu}_1 = 0
\end{equation}
\begin{equation}\label{e.relation-hat-i}
\widehat{\mu}_i + \widehat{\sigma}_i - \widehat{\nu}_1 - \widehat{\sigma}_k - \widehat{\zeta}_1 - \widehat{\mu}_k - \widehat{\zeta}_i - \widehat{\nu}_i = 0
\end{equation}
\begin{equation}\label{e.relation-hat-j}
\widehat{\mu}_j + \widehat{\sigma}_j - \widehat{\nu}_k + \widehat{\sigma}_1 - \widehat{\zeta}_k + \widehat{\mu}_1 - \widehat{\zeta}_j - \widehat{\nu}_j = 0
\end{equation}
\begin{equation}\label{e.relation-hat-k}
\widehat{\mu}_k + \widehat{\sigma}_k + \widehat{\nu}_j + \widehat{\sigma}_i + \widehat{\zeta}_j + \widehat{\mu}_i - \widehat{\zeta}_k - \widehat{\nu}_k = 0
\end{equation}

By adding together these four equations, we deduce that 
$$2\widehat{\sigma}_1 + 2\widehat{\sigma}_i - 2\widehat{\zeta}_1 - 2\widehat{\zeta}_k + 2\widehat{\mu}_1 + 2\widehat{\mu}_i - 2\widehat{\nu}_1 - 2\widehat{\nu}_k = 0,$$
that is, 
\begin{equation}\label{e.nu-hat-k}
\widehat{\nu}_k = \widehat{\sigma}_1 + \widehat{\sigma}_i - \widehat{\zeta}_1 - \widehat{\zeta}_k + \widehat{\mu}_1 + \widehat{\mu}_i - \widehat{\nu}_1
\end{equation}

By substituting \eqref{e.nu-hat-k} into \eqref{e.relation-hat-j} and \eqref{e.relation-hat-k}, we obtain that: 
\begin{equation}\label{e.mu-hat-j}
\widehat{\mu}_j  =  \widehat{\sigma}_i - \widehat{\sigma}_j - \widehat{\zeta}_1 + \widehat{\zeta}_j + \widehat{\mu}_i - \widehat{\nu}_1 + \widehat{\nu}_j
\end{equation}
\begin{equation}\label{e.mu-hat-k}
\widehat{\mu}_k = \widehat{\sigma}_1 - \widehat{\sigma}_k - \widehat{\zeta}_1 - \widehat{\zeta}_j + \widehat{\mu}_1 - \widehat{\nu}_1 - \widehat{\nu}_j
\end{equation}

Finally, by substituting \eqref{e.mu-hat-j} into \eqref{e.relation-hat-1} (or equivalently \eqref{e.mu-hat-k} into \eqref{e.relation-hat-i}), we get that: 
\begin{equation}\label{e.nu-hat-i}
\widehat{\nu}_i = - \widehat{\sigma}_1 + \widehat{\sigma}_i - \widehat{\zeta}_i + \widehat{\zeta}_j - \widehat{\mu}_1 + \widehat{\mu}_i + \widehat{\nu}_j
\end{equation}

\subsection{Dehn multitwist in the direction \texorpdfstring{$(1,1)$}{11}}
The straight lines in the direction $(1,1)$ decompose $\widetilde{L}$ into eight cylinders $a+, a-, b+, b-, c+, c-, d+, d-$ such that: 
\begin{itemize}
\item for each $\ast\in\{a, b, c, d\}$, the cylinder $\ast-$ is the image of $\ast+$ under the automorphism $-1\in Q$; 
\item the cylinders $a+$, $b+$ and $c+$ (resp.) cross the cycles $\mu_1$, $\nu_1$ and $\zeta_1$ (resp.) in $L_1$, and the cylinder $d+$ crosses the cycle $\nu_j$ in $L_j$.
\end{itemize}

Note that the ratio between the height and the width of each of these cylinders is $1/3$. Thus, the matrix $\underline{A}\in SL(2,\mathbb{Z})$ such that $\underline{A}(1,1) = (1,1)$ and $\underline{A}(1,0) = (1,0) + 3 (1,1) = (4,3)$, i.e., $\underline{A} = \left(\begin{array}{cc} 4 & -3 \\ 3 & -2\end{array}\right)$ belongs to the Veech group $SL(\widetilde{L})$. By a slight abuse of notation, we also call $\underline{A}$ the element of the affine group $\textrm{Aff}(\widetilde{L})$ with derivative $\underline{A}\in SL(2,\mathbb{Z})$ fixing pointwise the elements of $\Sigma$. 

For each  $\ast\in\{a, b, c, d\}$, let $\rho_{\ast+}$, resp. $\rho_{\ast-}$, be the homology class of the waist curve of the cylinder $\ast+$, resp. $\ast-$. Observe that the absolute cycle $\rho_{\ast-}$ is the image under the automorphism $-1\in Q$ of the absolute cycle $\rho_{\ast+}$. In particular, the cycles $\widehat{\rho}_{\ast} := \rho_{\ast+} - \rho_{\ast-}$ belong to $W_{\chi_2}$.

Denote by $A$ the action on the relative homology group $H_1(\widetilde{L}, \Sigma, \mathbb{R})$ induced by $\underline{A}$. We have that:   
\begin{eqnarray*}
& & A(\sigma_1) = \sigma_1 + \rho_{c-}, A(\sigma_i) = \sigma_i + \rho_{b-}, \\ 
& & A(\sigma_j) = \sigma_j + \rho_{a+}, A(\sigma_k) = \sigma_k + \rho_{d+}, \nonumber \\
& & A(\zeta_1) = \zeta_1 - \rho_{c+}, A(\zeta_i) = \zeta_i - \rho_{b+}, \nonumber \\ 
& & A(\zeta_j) = \zeta_j - \rho_{a-}, A(\zeta_k) = \zeta_k - \rho_{d-}, \nonumber \\ 
& & A(\mu_1) = \mu_1 + \rho_{a+}, A(\mu_i) = \mu_i + \rho_{d+}, \nonumber \\ 
& & A(\nu_1) = \nu_1 - \rho_{b+}, A(\nu_j) = \nu_j - \rho_{d+}. \nonumber
\end{eqnarray*}

Therefore, the action of $A$ on $W_{\chi_2}$ is described by the formulas: 
\begin{eqnarray}\label{e.A1}
& & A(\widehat{\sigma}_1) = \widehat{\sigma}_1 - \widehat{\rho}_{c}, A(\widehat{\sigma}_i) = \widehat{\sigma}_i - \widehat{\rho}_{b},  A(\widehat{\sigma}_j) = \widehat{\sigma}_j + \widehat{\rho}_{a}, A(\widehat{\sigma}_k) = \widehat{\sigma}_k + \widehat{\rho}_{d},  \\
& & A(\widehat{\zeta}_1) = \widehat{\zeta}_1 - \widehat{\rho}_{c}, A(\widehat{\zeta}_i) = \widehat{\zeta}_i - \widehat{\rho}_{b}, A(\widehat{\zeta}_j) = \widehat{\zeta}_j + \widehat{\rho}_{a}, A(\widehat{\zeta}_k) = \widehat{\zeta}_k + \widehat{\rho}_{d}, \nonumber \\ 
& & A(\widehat{\mu}_1) = \widehat{\mu}_1 + \widehat{\rho}_{a}, A(\widehat{\mu}_i) = \widehat{\mu}_i + \widehat{\rho}_{d}, A(\widehat{\nu}_1) = \widehat{\nu}_1 - \widehat{\rho}_{b}, A(\widehat{\nu}_j) = \widehat{\nu}_j - \widehat{\rho}_{d}. \nonumber
\end{eqnarray}

This information allows to write the matrix of $A$ in our preferred basis $\mathcal{B}$ once we express each $\widehat{\rho}_{\ast}$, $\ast\in\{a, b, c, d\}$ in terms of the elements of $\mathcal{B}$. In this direction, we note that the symmetry provided by automorphism $-1\in Q$ reduces this task to calculate $\rho_{\ast+}$, $\ast\in\{a, b, c, d\}$. 

A direct inspection of \autoref{f.1} reveals that 
$$\rho_{a+} = \mu_1 + \sigma_1 + \nu_i + \nu_{-k} + \zeta_{-k} + \mu_i$$ 
$$\rho_{b+} = \zeta_1 + \mu_j + \mu_k + \sigma_k + \nu_j + \nu_1$$ 
$$\rho_{c+} = \mu_j + \sigma_j + \nu_{-k} + \nu_{-i} + \zeta_{-i} + \mu_{-k}$$ 
$$\rho_{d+} = \zeta_j + \mu_{-1} + \mu_i + \sigma_i + \nu_{-1} + \nu_j$$ 

Hence, 
\begin{eqnarray*}
\widehat{\rho}_a & = & \widehat{\sigma}_1 - \widehat{\zeta}_k + \widehat{\mu}_1 + \widehat{\mu}_i + \widehat{\nu}_i - \widehat{\nu}_k \\ 
\widehat{\rho}_b & = & \widehat{\sigma}_k + \widehat{\zeta}_1 + \widehat{\mu}_j + \widehat{\mu}_k + \widehat{\nu}_1 + \widehat{\nu}_j \\ 
\widehat{\rho}_c & = & \widehat{\sigma}_j -  \widehat{\zeta}_i +  \widehat{\mu}_j -  \widehat{\mu}_k - \widehat{\nu}_i - \widehat{\nu}_k \\
\widehat{\rho}_d & = & \widehat{\sigma}_i + \widehat{\zeta}_j - \widehat{\mu}_1 + \widehat{\mu}_i -\widehat{\nu}_1 + \widehat{\nu}_j 
\end{eqnarray*}

By substituting \eqref{e.nu-hat-k}, \eqref{e.mu-hat-j}, \eqref{e.mu-hat-k} and \eqref{e.nu-hat-i} into the previous equations, we deduce that the formulas for $\widehat{\rho}_{\ast}$, $\ast\in\{a, b, c, d\}$ in terms of $\mathcal{B}$ are:

\begin{eqnarray}\label{e.rho-hat-B}
\widehat{\rho}_a & = & -\widehat{\sigma}_1 + \widehat{\zeta}_1 - \widehat{\zeta}_i + \widehat{\zeta}_j - \widehat{\mu}_1 + \widehat{\mu}_i + \widehat{\nu}_1 + \widehat{\nu}_j \\ 
\widehat{\rho}_b & = & \widehat{\sigma}_1 + \widehat{\sigma}_i - \widehat{\sigma}_j - \widehat{\zeta}_1 + \widehat{\mu}_1 + \widehat{\mu}_i - \widehat{\nu}_1 + \widehat{\nu}_j \nonumber \\ 
\widehat{\rho}_c & = & - \widehat{\sigma}_1 - \widehat{\sigma}_i + \widehat{\sigma}_k + \widehat{\zeta}_1 + \widehat{\zeta}_j + \widehat{\zeta}_k - \widehat{\mu}_1 - \widehat{\mu}_i +  \widehat{\nu}_1 + \widehat{\nu}_j \nonumber \\
\widehat{\rho}_d & = & \widehat{\sigma}_i + \widehat{\zeta}_j - \widehat{\mu}_1 + \widehat{\mu}_i -\widehat{\nu}_1 + \widehat{\nu}_j \nonumber 
\end{eqnarray}

By combining \eqref{e.A1} and \eqref{e.rho-hat-B}, we obtain that the matrix of $A$ in the basis $\mathcal{B}$ is:

\begin{equation}\label{e.A2} 
A=\left(\begin{array}{cccccccccccc}
2 & -1 & -1 & 0 & 1 & -1 & -1 & 0 & -1 & 0 & -1 & 0 \\
1 & 0 & 0 & 1 & 1 & -1 & 0 & 1 & 0 & 1 & -1 & -1 \\
0 & 1 & 1 & 0 & 0 & 1 & 0 & 0 & 0 & 0 & 1 & 0 \\
-1 & 0 & 0 & 1 & -1 & 0 & 0 & 0 & 0 & 0 & 0 & 0 \\
-1 & 1 & 1 & 0 & 0 & 1 & 1 & 0 & 1 & 0 & 1 & 0 \\
0 & 0 & -1 & 0 & 0 & 1 & -1 & 0 & -1 & 0 & 0 & 0 \\
-1 & 0 & 1 & 1 & -1 & 0 & 2 & 1 & 1 & 1 & 0 & -1 \\
-1 & 0 & 0 & 0 & -1 & 0 & 0 & 1 & 0 & 0 & 0 & 0 \\
1 & -1 & -1 & -1 & 1 & -1 & -1 & -1 & 0 & -1 & -1 & 1 \\
1 & -1 & 1 & 1 & 1 & -1 & 1 & 1 & 1 & 2 & -1 & -1 \\
-1 & 1 & 1 & -1 & -1 & 1 & 1 & -1 & 1 & -1 & 2 & 1 \\
-1 & -1 & 1 & 1 & -1 & -1 & 1 & 1 & 1 & 1 & -1 & 0
\end{array}\right)
\end{equation}

\subsection{Dehn multitwist in the direction \texorpdfstring{$(3,-1)$}{31}}
The straight lines in the direction $(3,-1)$ decompose $\widetilde{L}$ into eight cylinders $\alpha+, \alpha-, \beta+, \beta-, \gamma+, \gamma-, \delta+, \delta-$ such that: 
\begin{itemize}
\item for each $\ast\in\{\alpha, \beta, \gamma, \delta\}$, the cylinder $\ast-$ is the image of $\ast+$ under the automorphism $-1\in Q$; 
\item the cylinders $\alpha+$ and $\beta+$ (resp.) cross the cycles $\mu_1$ and $\sigma_1$ (resp.) in $L_1$, and the cylinder $\gamma+$, resp. $\delta+$, crosses the cycles $\nu_1$ and $\nu_i$, resp. $\zeta_1$ and $\nu_i$ in $L_1$.
\end{itemize}

Note that the matrix $\underline{B}\in SL(2,\mathbb{Z})$ such that $\underline{B}(3,-1) = (3,-1)$ and $\underline{B}(1,0) = (1,0) + 3 (3,-1) = (10,-3)$, i.e., $\underline{B} = \left(\begin{array}{cc} 10 & 27 \\ -3 & -8\end{array}\right)$ belongs to the Veech group $SL(\widetilde{L})$. By a slight abuse of notation, we also call $\underline{B}$ the element of the affine group $\textrm{Aff}(\widetilde{L})$ with derivative $\underline{B}\in SL(2,\mathbb{Z})$ fixing pointwise the elements of $\Sigma$. 

For each  $\ast\in\{\alpha, \beta, \gamma, \delta\}$, let $\rho_{\ast+}$, resp. $\rho_{\ast-}$, be the homology class of the waist curve of the cylinder $\ast+$, resp. $\ast-$. Observe that the absolute cycle $\rho_{\ast-}$ is the image under the automorphism $-1\in Q$ of the absolute cycle $\rho_{\ast+}$. In particular, the cycles $\widehat{\rho}_{\ast} := \rho_{\ast+} - \rho_{\ast-}$ belong to $W_{\chi_2}$.

Denote by $B$ the action on the relative homology group $H_1(\widetilde{L}, \Sigma, \mathbb{R})$ induced by $\underline{B}$. We have that:
\begin{eqnarray*}
& & B(\sigma_1) = \sigma_1 + \rho_{\beta+}, B(\sigma_i) = \sigma_i + \rho_{\delta+}\\
& & B(\sigma_j) = \sigma_j + \rho_{\alpha-}, B(\sigma_k) = \sigma_k + \rho_{\gamma-}, \nonumber \\
& & B(\zeta_1) = \zeta_1 + \rho_{\delta+} + \rho_{\beta-} + \rho_{\delta-}, B(\zeta_i) = \zeta_i + \rho_{\beta-} + \rho_{\delta-} + \rho_{\beta+}, \nonumber \\ 
& & B(\zeta_j) = \zeta_j + \rho_{\gamma+} + \rho_{\alpha+} + \rho_{\gamma-}, B(\zeta_k) = \zeta_k + \rho_{\alpha-} + \rho_{\gamma+} + + \rho_{\alpha+}, \nonumber \\ 
& & B(\mu_1) = \mu_1 + \rho_{\alpha+}, B(\mu_i) = \mu_i + \rho_{\gamma+}, \nonumber \\ 
& & B(\nu_1) = \nu_1 + \rho_{\alpha+} + \rho_{\beta+} + \rho_{\gamma+}, B(\nu_j) = \nu_j + \rho_{\beta+} + \rho_{\alpha-} + \rho_{\delta-}. \nonumber
\end{eqnarray*}

Therefore, the action of $B$ on $W_{\chi_2}$ is described by the formulas: 
\begin{eqnarray}\label{e.B1}
& & B(\widehat{\sigma}_1) = \widehat{\sigma}_1 + \widehat{\rho}_{\beta}, B(\widehat{\sigma}_i) = \widehat{\sigma}_i + \widehat{\rho}_{\delta}, \\ 
& & B(\widehat{\sigma}_j) = \widehat{\sigma}_j - \widehat{\rho}_{\alpha}, B(\widehat{\sigma}_k) = \widehat{\sigma}_k - \widehat{\rho}_{\gamma},  \nonumber \\
& & B(\widehat{\zeta}_1) = \widehat{\zeta}_1 - \widehat{\rho}_{\beta}, B(\widehat{\zeta}_i) = \widehat{\zeta}_i - \widehat{\rho}_{\delta}, \nonumber \\ 
& & B(\widehat{\zeta}_j) = \widehat{\zeta}_j + \widehat{\rho}_{\alpha}, B(\widehat{\zeta}_k) = \widehat{\zeta}_k + \widehat{\rho}_{\gamma}, \nonumber \\ 
& & B(\widehat{\mu}_1) = \widehat{\mu}_1 + \widehat{\rho}_{\alpha}, B(\widehat{\mu}_i) = \widehat{\mu}_i + \widehat{\rho}_{\gamma}, \nonumber \\ 
& & B(\widehat{\nu}_1) = \widehat{\nu}_1 + \widehat{\rho}_{\alpha} + \widehat{\rho}_{\beta} + \widehat{\rho}_{\gamma}, B(\widehat{\nu}_j) = \widehat{\nu}_j - \widehat{\rho}_{\alpha} + \widehat{\rho}_{\beta} - \widehat{\rho}_{\delta}. \nonumber
\end{eqnarray}

This permits to write the matrix of $B$ in our preferred basis $\mathcal{B}$ once we express each $\widehat{\rho}_{\ast}$, $\ast\in\{\alpha, \beta, \gamma, \delta\}$ in terms of the elements of $\mathcal{B}$.
In this direction, we note that the symmetry provided by automorphism $-1\in Q$ reduces this task to calculate $\rho_{\ast+}$, $\ast\in\{\alpha, \beta, \gamma, \delta\}$. 

A direct inspection of \autoref{f.1} shows that 
\begin{align*}
\rho_{\alpha+} = &\sigma_i + \sigma_{-i} + \sigma_{-j} - \zeta_{-j} + \mu_1 + \mu_{-1} + \mu_i + 2\mu_{-i} + \mu_{-j} - \nu_{-i} - \nu_{-j}\\
\rho_{\beta+} = &\sigma_1 + \sigma_{k} + \sigma_{-k} - \zeta_{1} + \mu_1 + \mu_{j} + \mu_{-j} + 2\mu_{k} + \mu_{-k} - \nu_{1} - \nu_{k}\\
\rho_{\gamma+} = &\sigma_1 + \sigma_{-1} + \sigma_{-k} - \zeta_{-k} + 2\mu_1 + \mu_{-1} + \mu_{i} + \mu_{-i} + \mu_{-k} - \nu_{1} - \nu_{-k}\\
\rho_{\delta+} = &\sigma_i + \sigma_{j} + \sigma_{-j} - \zeta_{i} + \mu_i + \mu_{j} + 2\mu_{-j} + \mu_{k} + \mu_{-k} - \nu_{i} - \nu_{-j}
\end{align*}
Hence, 
\begin{eqnarray*}
\widehat{\rho}_{\alpha} & = & - \widehat{\sigma}_j + \widehat{\zeta}_j - \widehat{\mu}_i - \widehat{\mu}_j + \widehat{\nu}_i + \widehat{\nu}_j \\ 
\widehat{\rho}_{\beta} & = & \widehat{\sigma}_1 - \widehat{\zeta}_1 + \widehat{\mu}_1 + \widehat{\mu}_k - \widehat{\nu}_1 - \widehat{\nu}_k \\ 
\widehat{\rho}_{\gamma} & = & - \widehat{\sigma}_k + \widehat{\zeta}_k + \widehat{\mu}_1 -  \widehat{\mu}_k - \widehat{\nu}_1 + \widehat{\nu}_k \\
\widehat{\rho}_{\delta} & = & \widehat{\sigma}_i - \widehat{\zeta}_i + \widehat{\mu}_i - \widehat{\mu}_j -\widehat{\nu}_i + \widehat{\nu}_j 
\end{eqnarray*}

By substituting \eqref{e.nu-hat-k}, \eqref{e.mu-hat-j}, \eqref{e.mu-hat-k} and \eqref{e.nu-hat-i} into the previous equations, we deduce that the formulas for $\widehat{\rho}_{\ast}$, $\ast\in\{\alpha, \beta, \gamma, \delta\}$ in terms of $\mathcal{B}$ are:

\begin{eqnarray}\label{e.rho-hat-B-2}
\widehat{\rho}_{\alpha} & = & -\widehat{\sigma}_1 + \widehat{\zeta}_1 - \widehat{\zeta}_i + \widehat{\zeta}_j - \widehat{\mu}_1 - \widehat{\mu}_i + \widehat{\nu}_1 + \widehat{\nu}_j \\ 
\widehat{\rho}_{\beta} & = & \widehat{\sigma}_1 - \widehat{\sigma}_i - \widehat{\sigma}_k - \widehat{\zeta}_1 - \widehat{\zeta}_j + \widehat{\zeta}_k + \widehat{\mu}_1 - \widehat{\mu}_i - \widehat{\nu}_1 - \widehat{\nu}_j \nonumber \\ 
\widehat{\rho}_{\gamma} & = & \widehat{\sigma}_i + \widehat{\zeta}_j + \widehat{\mu}_1 + \widehat{\mu}_i -  \widehat{\nu}_1 + \widehat{\nu}_j \nonumber \\
\widehat{\rho}_{\delta} & = & \widehat{\sigma}_1 - \widehat{\sigma}_i + \widehat{\sigma}_j + \widehat{\zeta}_1 - 2\widehat{\zeta}_j + \widehat{\mu}_1 - \widehat{\mu}_i +\widehat{\nu}_1 - \widehat{\nu}_j \nonumber 
\end{eqnarray}

By combining \eqref{e.B1} and \eqref{e.rho-hat-B-2}, we see that the matrix of $B$ in the basis $\mathcal{B}$ is:
\begin{align}\label{e.B2} 
B=\left(\begin{array}{cccccccccccc}
2 & 1 & 1 & 0 & -1 & -1 & -1 & 0 & -1 & 0 & 0 & 1 \\
-1 & 0 & 0 & -1 & 1 & 1 & 0 & 1 & 0 & 1 & 0 & 0 \\
0 & 1 & 1 & 0 & 0 & -1 & 0 & 0 & 0 & 0 & 0 & -1 \\
-1 & 0 & 0 & 1 & 1 & 0 & 0 & 0 & 0 & 0 & -1 & -1 \\
-1 & 1 & -1 & 0 & 2 & -1 & 1 & 0 & 1 & 0 & 0 & -3 \\
0 & 0 & 1 & 0 & 0 & 1 & -1 & 0 & -1 & 0 & -1 & 1 \\
-1 & -2 & -1 & -1 & 1 & 2 & 2 & 1 & 1 & 1 & 1 & 0 \\
1 & 0 & 0 & 0 & -1 & 0 & 0 & 1 & 0 & 0 & 1 & 1 \\
1 & 1 & 1 & -1 & -1 & -1 & -1 & 1 & 0 & 1 & 1 & 1 \\
-1 & -1 & 1 & -1 & 1 & 1 & -1 & 1 & -1 & 2 & -1 & 1 \\
-1 & 1 & -1 & 1 & 1 & -1 & 1 & -1 & 1 & -1 & 0 & -3 \\
-1 & -1 & -1 & -1 & 1 & 1 & 1 & 1 & 1 & 1 & 1 & 0
\end{array}\right) 
\end{align}

\subsection{Dehn multitwist in the direction \texorpdfstring{$(-1,3)$}{-13}}
The origami $\widetilde{L}$ has a ``symmetry'' given by the (orientation-reversing) matrix $\underline{R}=\left(\begin{array}{cc} 0 & 1 \\ 1 & 0 \end{array}\right)\in GL(2,\mathbb{Z})$. 

More precisely, by applying this matrix to each $L_g$, $g\in Q$, in \autoref{f.1}, we exchange the roles of $\mu_g$, $\mu_{gj}$, resp. $\sigma_g$, $\sigma_{gj}$, and $\nu_g$, $\nu_{gi}$, resp. $\zeta_g$, $\zeta_{gi}$.
By relabelling the sides $\nu_g$, $\zeta_g$, $\mu_g$, $\sigma_g$ (resp.) of $\underline{R}(L_g)$ as $\mu_{\phi(g)}$, $\sigma_{\phi(g)}$, $\nu_{\phi(g)}$, $\zeta_{\phi(g)}$ (resp.) where $\phi:Q\to Q$ is the (outer) automorphism of $Q$ with $\phi(i)=j$, $\phi(j)=i$, $\phi(k)=-k$, we recover the origami $\widetilde{L}$. 

Note that this ``symmetry'' exchanges the roles of the directions $(3,-1)$ and $(-1,3)$ in the origami $\widetilde{L}$. In particular, the eight cylinders $\ast+$, $\ast-$ with $\ast\in\{\alpha,\beta,\gamma,\delta\}$ of $\widetilde{L}$ in the direction $(3,-1)$ are associated under this ``symmetry'' to eight cylinders $\ast+$, $\ast-$ with $\ast\in\{\overline{\alpha}, \overline{\beta}, \overline{\gamma}, \overline{\delta}\}$. 

The matrix $\underline{C}=\left(\begin{array}{cc} -8 & -3 \\ 27 & 10 \end{array}\right)$ deduced from $\underline{B}$ by conjugation with $\underline{R}$ belongs to the Veech group $SL(\widetilde{L})$. Denote by $C$ the action on homology of the affine homeomorphism with linear part $\underline{C}$ and fixing pointwise $\Sigma$.

For each $\ast\in\{\overline{\alpha}, \overline{\beta}, \overline{\gamma}, \overline{\delta}\}$, let $\rho_{\ast+}$, resp. $\rho_{\ast-}$, be the absolute homology classes of the waist curves of the cylinder $\ast+$, resp. $\ast-$, so that the cycles $\widehat{\rho}_{\ast}:=\rho_{\ast+} - \rho_{\ast-}$ belong to $W_{\chi_2}$. 

By ``symmetry'', we deduce from the analogous formulas for $B$ that the action of $C$ on $W_{\chi_2}$ is:

\begin{eqnarray*}
& & C(\widehat{\zeta}_1) = \widehat{\zeta}_1 + \widehat{\rho}_{\overline{\beta}}, C(\widehat{\zeta}_j) = \widehat{\zeta}_j + \widehat{\rho}_{\overline{\delta}}, \\ 
& & C(\widehat{\zeta}_i) = \widehat{\zeta}_i - \widehat{\rho}_{\overline{\alpha}}, C(\widehat{\zeta}_k) = \widehat{\zeta}_k + \widehat{\rho}_{\overline{\gamma}},  \nonumber \\
& & C(\widehat{\sigma}_1) = \widehat{\sigma}_1 - \widehat{\rho}_{\overline{\beta}}, C(\widehat{\sigma}_j) = \widehat{\sigma}_j - \widehat{\rho}_{\overline{\delta}}, \nonumber \\ 
& & C(\widehat{\sigma}_i) = \widehat{\sigma}_i + \widehat{\rho}_{\overline{\alpha}}, C(\widehat{\sigma}_k) = \widehat{\sigma}_k - \widehat{\rho}_{\overline{\gamma}}, \nonumber \\ 
& & C(\widehat{\nu}_1) = \widehat{\nu}_1 + \widehat{\rho}_{\overline{\alpha}}, C(\widehat{\nu}_j) = \widehat{\nu}_j + \widehat{\rho}_{\overline{\gamma}}, \nonumber \\ 
& & C(\widehat{\mu}_1) = \widehat{\mu}_1 + \widehat{\rho}_{\overline{\alpha}} + \widehat{\rho}_{\overline{\beta}} + \widehat{\rho}_{\overline{\gamma}}, C(\widehat{\mu}_i) = \widehat{\mu}_i - \widehat{\rho}_{\overline{\alpha}} + \widehat{\rho}_{\overline{\beta}} - \widehat{\rho}_{\overline{\delta}}. \nonumber
\end{eqnarray*}
where 

\begin{eqnarray*}
\widehat{\rho}_{\overline{\alpha}} & = &  - \widehat{\zeta}_1+ \widehat{\sigma}_1 - \widehat{\sigma}_j + \widehat{\sigma}_i - \widehat{\nu}_1 - \widehat{\nu}_j + \widehat{\mu}_1 + \widehat{\mu}_i \\ 
\widehat{\rho}_{\beta} & = & \widehat{\zeta}_1 - \widehat{\zeta}_j + \widehat{\zeta}_k - \widehat{\sigma}_1 - \widehat{\sigma}_i - \widehat{\sigma}_k + \widehat{\nu}_1 - \widehat{\nu}_j - \widehat{\mu}_1 - \widehat{\mu}_i \nonumber \\ 
\widehat{\rho}_{\gamma} & = &  \widehat{\zeta}_j + \widehat{\sigma}_i + \widehat{\nu}_1 + \widehat{\nu}_j -  \widehat{\mu}_1 + \widehat{\mu}_i \nonumber \\
\widehat{\rho}_{\delta} & = & \widehat{\zeta}_1 - \widehat{\zeta}_j + \widehat{\zeta}_i + \widehat{\sigma}_1 - 2\widehat{\sigma}_i + \widehat{\nu}_1 - \widehat{\nu}_j +\widehat{\mu}_1 - \widehat{\mu}_i \nonumber 
\end{eqnarray*}

Therefore, the matrix of $C$ in the basis $\mathcal{B}$ is:

\begin{equation}\label{e.C} 
C=\left(\begin{array}{cccccccccccc}
2 & 1 & -1 & 0 & -1 & -1 & 1 & 0 & 0 & -3 & 1 & 0 \\
1 & 2 & 2 & -1 & -1 & -1 & -2 & 1 & 1 & 0 & 1 & 1 \\
0 & -1 & 1 & 0 & 0 & 1 & 0 & 0 & -1 & 1 & -1 & 0 \\
1 & 0 & 0 & 1 & -1 & 0 & 0 & 0 & -1 & -1 & 0 & 0 \\
-1 & -1 & -1 & 0 & 2 & 1 & 1 & 0 & 0 & 1 & -1 & 0 \\
0 & 0 & -1 & 0 & 0 & 1 & 1 & 0 & 0 & -1 & 0 & 0 \\
1 & 0 & 1 & -1 & -1 & 0 & 0 & 1 & 0 & 0 & 0 & 1 \\
-1 & 0 & 0 & 0 & 1 & 0 & 0 & 1 & 1 & 1 & 0 & 0 \\
1 & 1 & -1 & 1 & -1 & -1 & 1 & -1 & 0 & -3 & 1 & -1 \\
1 & 1 & 1 & -1 & -1 & -1 & -1 & 1 & 1 & 0 & 1 & 1 \\
-1 & -1 & -1 & -1 & 1 & 1 & 1 & 1 & 1 & 1 & 0 & 1 \\
1 & -1 & 1 & -1 & -1 & 1 & -1 & 1 & -1 & 1 & -1 & 2
\end{array}\right) 
\end{equation}

\subsection{End of proof of \autoref{t.FFM2}} By \autoref{t.FFM1}, the proof of \autoref{t.FFM2} is complete once we show that $W_{\chi_2}$ can not be decomposed as 
$$W_{\chi_2} = U\oplus V$$
for some symplectic $\textrm{Aff}_{***}(\widetilde{L})$-invariant $\textrm{Aut}(\widetilde{L})$-submodules $U\simeq 2\chi_2$ and $V\simeq\chi_2$, where $\textrm{Aff}_{***}(\widetilde{L})$ is a finite-index subgroup of $\textrm{Aff}(\widetilde{L})$.

Suppose by contradiction that $W_{\chi_2}$ admits such a decomposition $W_{\chi_2}=U\oplus V$. Recall from Subsection \ref{ss.dichotomy} that the action of $\textrm{Aff}_{***}(\widetilde{L})$ on $V\simeq \chi_2$ occurs through a subgroup of the compact group $SO^*(2)$. Thus, $V$ is the central eigenspace of the action of any element of $\textrm{Aff}_{***}(\widetilde{L})$ on $W_{\chi_2}$ whose spectrum is ``simple'' (i.e., it has three eigenvalues of multiplicity four whose moduli are distinct). 

Hence, we reach a contradiction if there are two elements of $\textrm{Aff}_{***}(\widetilde{L})$ acting on $W_{\chi_2}$ with simple spectrum such that their central eigenspaces are distinct.  

We claim that, for some $k, l\in\mathbb{N}$, the matrices $(A.B)^k$ and $(C.B)^l$ associated to the actions on $W_{\chi_2}$ of the powers $(\underline{A}\circ\underline{B})^k$ and $(\underline{C}\circ\underline{B})^l$ of the affine homeomorphisms $\underline{A}\circ\underline{B}$ and $\underline{C}\circ\underline{B}$ have the desired properties. 

Indeed, a direct computation with \eqref{e.A2} and \eqref{e.B2} shows that the matrix 
\begin{equation*}
A.B=\left(\begin{array}{cccccccccccc}
5 & 2 & 0 & 2 & -2 & -4 & -2 & -2 & -2 & -2 & -1 & 1 \\
2 & 1 & 2 & 0 & 0 & -2 & -2 & 2 & -2 & 2 & -1 & 1 \\
-2 & 2 & 1 & 0 & 2 & 0 & 0 & 0 & 0 & 0 & -1 & -3 \\
-2 & -2 & 0 & 1 & 0 & 2 & 0 & 0 & 0 & 0 & -1 & 1 \\
-4 & 0 & 0 & -2 & 3 & 2 & 2 & 2 & 2 & 2 & 1 & -3 \\
0 & 0 & 0 & 2 & 0 & 1 & -2 & -2 & -2 & -2 & -3 & 1 \\
-2 & -4 & 2 & -2 & 0 & 4 & 1 & 4 & 0 & 4 & 1 & 3 \\
0 & -2 & 0 & 0 & -2 & 2 & 0 & 1 & 0 & 0 & 1 & 3 \\
4 & 2 & -2 & 0 & -2 & -4 & 0 & -2 & 1 & -2 & 2 & 0 \\
2 & 0 & 4 & -2 & 0 & -2 & -2 & 4 & -2 & 5 & 0 & 2 \\
-4 & 0 & -2 & -2 & 2 & 2 & 4 & 0 & 4 & 0 & 3 & -4 \\
0 & -4 & 2 & -2 & -2 & 2 & 0 & 4 & 0 & 4 & 2 & 5
\end{array}\right) 
\end{equation*}
has characteristic polynomial 
\begin{eqnarray*}
P_{AB}(x) &=& x^{12} - 28 x^{11} + 322 x^{10} -  1964 x^9 + 6895 x^8 - 14392 x^7 \\ &+& 18332 x^6  -14392 x^5 +6895 x^4 -1964 x^3 +322 x^2 -28 x + 1 \\ &=&(x-1)^4 (x^2-6x+1)^4
\end{eqnarray*}
In particular, $A.B$ has three eigenvalues (each of them with multiplicity four), namely,  $3+2\sqrt{2}$, $1$ and $3-2\sqrt{2}$, so that $A.B$ has ``simple'' spectrum. Furthermore, it is not hard to see that the central eigenspace $V_{AB}$  of $A.B$ (associated to the eigenvalue $1$) is spanned by the following four vectors:
$$v_{AB}^{(1)} = (-1, 1, -1, 1, 1, -1, 1, 1, 0, 0, 0, 2),$$
$$v_{AB}^{(2)} = (-1, -1, 1, 1, 1, -1, -1, -1, 0, 0, 2, 0),$$
$$v_{AB}^{(3)} = (0, 0, 0, 0, 0, 0, 0, -1, 0, 1, 0, 0),$$
$$v_{AB}^{(4)} = (0, 0, 0, 0, 0, 0, -1, 0, 1, 0, 0, 0)$$

Similarly, a immediate calculation with \eqref{e.B2} and \eqref{e.C} reveals that the matrix 
\begin{equation*}
C.B=\left(\begin{array}{cccccccccccc}
5 & 2 & -4 & 2& -4& -2& 4& -2& 3& -5& 5& -1 \\
4& 7& 4& -2& -4& -6& -4& 2& -1& 1& 3& 1 \\
0& -2& 3& 0& 0& 2& -2& 0& -3& 1& -3& 3 \\
2& 0& 0& 3& -2& 0& 0& -2& -1& -3& -1& 1 \\
-4& -4& -2& -2& 5& 4& 2& 2& 1& 3& -1& -1 \\
0& -2& -2& 0& 0& 3& 2& 0& 1& -1& 1& 1 \\
4& 0& 2& -2& -4& 0& -1& 2& -1& 1& 3& 5 \\
-2& 0& 0& -2& 2& 0& 0& 3& 1& 3& 1& -1 \\
2& 2& -4& 4& -2& -2& 4& -4& 3& -6& 2& -4 \\
4& 4& 2& -2& -4& -4& -2& 2& 0& 1& 4& 2 \\
-2& -4& -2& -4& 2& 4& 2& 4& 2& 4& 3& 2 \\
4& -2& 4& -2& -4& 2& -4& 2& -4& 2& 0& 9
\end{array}\right) 
\end{equation*}
has characteristic polynomial 
\begin{eqnarray*}
P_{CB}(x) &=& x^{12} - 44 x^{11} + 770 x^{10} - 6780 x^9 + 31471 x^8 - 76120 x^7 \\ &+& + 101404 x^6  - 76120 x^5 + 31471 x^4 - 6780 x^3 + 770 x^2 - 44 x + 1 \\ &=& (x-1)^4 (x^2-10x+1)^4
\end{eqnarray*}
In particular, $C.B$ has three eigenvalues (each of them with multiplicity four), namely,  $5+2\sqrt{6}$, $1$ and $5-2\sqrt{6}$, so that $C.B$ also has ``simple'' spectrum. Furthermore, it is not hard to see that the central eigenspace $V_{CB}$  of $C.B$ (associated to the eigenvalue $1$) is spanned by the following four vectors:
$$v_{CB}^{(1)} = (0, 0, 0, 1, 0, 0, 0, 1, 0, 0, 0, 0),$$
$$v_{CB}^{(2)} = (0, 0, 1, 0, 0, 0, 1, 0, 0, 0, 0, 0),$$
$$v_{CB}^{(3)} = (0, 1, 0, 0, 0, 1, 0, 0, 0, 0, 0, 0),$$
$$v_{CB}^{(4)} = (1, 0, 0, 0, 1, 0, 0, 0, 0, 0, 0, 0)$$

It follows that $V_{AB}$ and $V_{CB}$ are distinct (and, actually, $\left\{v_{AB}^{(n)}, v_{CB}^{(m)}\right\}_{1\leq n, m\leq 4}$ span a $8$-dimensional subspace). Moreover, the same properties are true for any powers $(A.B)^k$ and $(C.B)^l$: the matrices $(A.B)^k$ and $(C.B)^l$ have ``simple spectrum'' and $V_{(A.B)^k} = V_{A.B}$ and $V_{(C.B)^l} = V_{C.B}$ are distinct for all $k, l\in\mathbb{N}$. By taking $k, l\in \mathbb{N}$ so that $(\underline{A}\circ\underline{B})^k$ and $(\underline{C}\circ\underline{B})^l$ belong to $\textrm{Aff}_{***}(\widetilde{L})$ (this is always possible because $\textrm{Aff}_{***}(\widetilde{L})$ is a finite-index subgroup of $\textrm{Aff}(\widetilde{L})$), we conclude that there is no decomposition $W_{\chi_2} = U\oplus V$ with the features described in the beginning of this subsection. This completes the proof of \autoref{t.FFM2}. \hfill \qed

\section*{Acknowledgments}

S.F. is grateful to the other two authors for interesting discussions and questions on this topic.
He is also grateful to Alex Eskin and Anton Zorich for discussions related to this topic.

C.M. is thankful to Jean-Christophe Yoccoz for some discussions related to \cite{MYZ}. 

C.M. was partially supported by the French ANR grant ``GeoDyM'' (ANR-11-BS01-0004) and by the Balzan Research Project of J. Palis.

\end{document}

%% file: FFM-1.pdf_tex

\begingroup
  \makeatletter
  \providecommand\color[2][]{%
    \errmessage{(Inkscape) Color is used for the text in Inkscape, but the package 'color.sty' is not loaded}
    \renewcommand\color[2][]{}%
  }
  \providecommand\transparent[1]{%
    \errmessage{(Inkscape) Transparency is used (non-zero) for the text in Inkscape, but the package 'transparent.sty' is not loaded}
    \renewcommand\transparent[1]{}%
  }
  \providecommand\rotatebox[2]{#2}
  \ifx\svgwidth\undefined
    \setlength{\unitlength}{292.9578125pt}
  \else
    \setlength{\unitlength}{\svgwidth}
  \fi
  \global\let\svgwidth\undefined
  \makeatother
  \begin{picture}(1,0.68549493)%
    \put(0,0){\includegraphics[width=\unitlength]{FFM-1.pdf}}%
    \put(0.06293568,0.32760423){\color[rgb]{0,0,0}\makebox(0,0)[b]{\smash{$L_g$}}}%
    \put(0.22519241,0.1814681){\color[rgb]{0,0,0}\makebox(0,0)[b]{\smash{$\nu_g$}}}%
    \put(0.23224334,0.45712906){\color[rgb]{0,0,0}\makebox(0,0)[b]{\smash{$\zeta_g$}}}%
    \put(0.40196267,0.64472006){\color[rgb]{0,0,0}\makebox(0,0)[b]{\smash{$\mu_{gj}$}}}%
    \put(0.38997289,0.01192044){\color[rgb]{0,0,0}\makebox(0,0)[b]{\smash{$\mu_g$}}}%
    \put(0.68643096,0.37374728){\color[rgb]{0,0,0}\makebox(0,0)[b]{\smash{$\sigma_{gj}$}}}%
    \put(0.86832575,0.18161477){\color[rgb]{0,0,0}\makebox(0,0)[b]{\smash{$\nu_{gi}$}}}%
    \put(0.67444118,0.01192044){\color[rgb]{0,0,0}\makebox(0,0)[b]{\smash{$\sigma_g$}}}%
    \put(0.59797966,0.45727573){\color[rgb]{0,0,0}\makebox(0,0)[b]{\smash{$\zeta_{gi}$}}}%
    \put(0.23982317,0.0500772){\color[rgb]{0,0,0}\makebox(0,0)[b]{\smash{$\overline{g}$}}}%
  \end{picture}%
\endgroup

%% file: FFM-2.pdf_tex

\begingroup
  \makeatletter
  \providecommand\color[2][]{%
    \errmessage{(Inkscape) Color is used for the text in Inkscape, but the package 'color.sty' is not loaded}
    \renewcommand\color[2][]{}%
  }
  \providecommand\transparent[1]{%
    \errmessage{(Inkscape) Transparency is used (non-zero) for the text in Inkscape, but the package 'transparent.sty' is not loaded}
    \renewcommand\transparent[1]{}%
  }
  \providecommand\rotatebox[2]{#2}
  \ifx\svgwidth\undefined
    \setlength{\unitlength}{300pt}
  \else
    \setlength{\unitlength}{\svgwidth}
  \fi
  \global\let\svgwidth\undefined
  \makeatother
  \begin{picture}(1,0.79989045)%
    \put(0,0){\includegraphics[width=\unitlength]{FFM-2.pdf}}%
    \put(0.50082057,0.66745402){\color[rgb]{0,0,0}\makebox(0,0)[b]{\smash{$M_3$}}}%
    \put(0.50082057,0.52903764){\color[rgb]{0,0,0}\makebox(0,0)[b]{\smash{$M_1$}}}%
    \put(0.77755432,0.61769377){\color[rgb]{0,0,0}\makebox(0,0)[b]{\smash{$M_2$}}}%
    \put(0.85607832,0.38972087){\color[rgb]{0,0,0}\makebox(0,0)[b]{\smash{$M_{11}$}}}%
    \put(0.70916248,0.3883789){\color[rgb]{0,0,0}\makebox(0,0)[b]{\smash{$M_5$}}}%
    \put(0.28614607,0.3883789){\color[rgb]{0,0,0}\makebox(0,0)[b]{\smash{$M_4$}}}%
    \put(0.14429626,0.38972087){\color[rgb]{0,0,0}\makebox(0,0)[b]{\smash{$M_{10}$}}}%
    \put(0.78008735,0.16681406){\color[rgb]{0,0,0}\makebox(0,0)[b]{\smash{$M_8$}}}%
    \put(0.21775419,0.16681406){\color[rgb]{0,0,0}\makebox(0,0)[b]{\smash{$M_6$}}}%
    \put(0.50082057,0.11741993){\color[rgb]{0,0,0}\makebox(0,0)[b]{\smash{$M_7$}}}%
    \put(0.50082057,0.25800319){\color[rgb]{0,0,0}\makebox(0,0)[b]{\smash{$M_9$}}}%
    \put(0.23801845,0.61903574){\color[rgb]{0,0,0}\makebox(0,0)[b]{\smash{$\widetilde{L}$}}}%
  \end{picture}%
\endgroup